\documentclass[11pt]{article}
\usepackage{latexsym}
\usepackage{amssymb}
\usepackage{amsmath}
\usepackage{amsfonts}
\usepackage{textcomp}
\usepackage{amscd}
\usepackage{amsthm}
\usepackage{mathrsfs}
\usepackage{amsxtra}
\usepackage{stmaryrd}
\usepackage[all]{xy}
\usepackage[dvips]{graphicx}

\numberwithin{equation}{section}
\newtheorem{theorem}{Theorem}[section]
\newtheorem{conjecture}[theorem]{Conjecture}
\newtheorem{corollary}[theorem]{Corollary}
\newtheorem{lemma}[theorem]{Lemma}
\newtheorem{proposition}[theorem]{Proposition}

\theoremstyle{definition}

\newtheorem{definition}[theorem]{Definition}
\newtheorem{example}[theorem]{Example}

\newtheorem{remark}[theorem]{Remark}

\def\Z{{\mathbb {Z}}_p}
\def\Q{{\mathbb {Q}}_p}

\def\G{{\rm GL}_2(\Q)}

\def\B{{\rm B}(\Q)}

\def\1z{\mathrm{1}_{\mathbb{Z}_p}}
\def\w{\left(\begin{smallmatrix}
 0&1 \\
 1&0
\end{smallmatrix}\right)}
\def\(m(V)){m(\alpha,\beta)}
\def\ra{\rightarrow}

\def\Qp{{\mathbb{Q}}_p}

\def\eps{\varepsilon}
\def\OO{\mathcal{O}}

\def\val{\mathrm{val}}
\def\aa{\textbf{A}}
\def\bb{\textbf{B}}

\def\oe{\OO_{\calE_L}}

\def\calE{\mathcal{E}}

\def\rig{\mathrm{rig}}

\def\res{\mathrm{res}}
\def\Res{\mathrm{Res}}
\def\D{\mathrm{D}}
\def\A{\mathcal{A}}
\DeclareMathOperator{\rank}{rank} \DeclareMathOperator{\Ind}{Ind}

\def\bcris{\textbf{B}_{\mathrm{cris}}}

\def\bdR{\textbf{B}_{\rm dR}}
\def\m{(\varphi,\Gamma)}

\def\r{\mathcal{R}}

\title{Locally analytic vectors of some crystabelian representations of $\G$}
\author{Ruochuan Liu\\University of Michigan\\
ruochuan@umich.edu}

\begin{document}

\maketitle

\begin{abstract}
For $V$ a $2$-dimensional $p$-adic representation of $G_{\Q}$, we denote by $\mathrm{B}(V)$ the admissible unitary representation of $\G$ attached to $V$ under the $p$-adic local Langlands correspondence of $\G$ initiated by Breuil. In this article, building on the works of Berger-Breuil and Colmez, we determine the locally analytic vectors $\mathrm{B}(V)_{\mathrm{an}}$ of $\mathrm{B}(V)$  when $V$ is irreducible, crystabelian and Frobenius semi-simple with distinct Hodge-Tate weights; this proves a conjecture of Breuil. Using this result, we verify Emerton's conjecture that $\dim \mathrm{Ref}^{\eta\otimes\psi}(V)=\dim\mathrm{Exp}^{\eta|\cdot|\otimes x\psi }(\mathrm{B}(V)_{\mathrm{an}}\otimes(x|\cdot|\circ\det))$ for those $V$ which are irreducible, crystabelian and Frobenius semi-simple.
\end{abstract}

\tableofcontents
\section*{Introduction}
Fix a prime $p>2$, as well as a finite extension $L$ of $\Q$ to be the coefficient field. Recall that for any integer $k\geq2$, the set of $2$-dimensional semistable and non-crystalline $L$-linear representations of $G_{\Q}$ with Hodge-Tate weights $(0,k-1)$ is parameterized by $L$ via the $\mathscr{L}$-invariant. For any $\mathscr{L}\in L$, we denote by $V(k,\mathscr{L})$ the $L$-linear representation of $G_{\Q}$ corresponding to $\mathscr{L}$. In \cite{Br04}, Breuil constructed a family of locally analytic representations $(\Sigma(k,\mathscr{L}))$ of $\G$ associated to the family of $L$-linear representations $(V(k,\mathscr{L}))$ of $G_{\Q}$ for all $\mathscr{L}\in L$. Breuil's work suggested the possible existence of a $p$-adic version of the local Langlands correspondence for $\G$. In fact, Breuil conjectured that there should be a $p$-adic local Langlands correspondence for $\G$ which attaches to any $2$-dimensional potentially semistable $L$-linear representation of $G_{\Q}$ a $p$-adic admissible unitary representation of $\G$. Thanks to much recent work, especially that of Colmez, one can now extend this conjecture to all $2$-dimensional $L$-linear representations of $G_{\Q}$; we denote this correspondence by $V\mapsto\mathrm{B}(V)$. Although the present version of the $p$-adic local Langlands correspondence for $\G$ is formulated at the level of Banach space representations, it is very useful, as in Breuil's work (and many other examples), to have the information of the space of locally analytic vectors $\mathrm{B}(V)_{\mathrm{an}}$ of $\mathrm{B}(V)$. This is the theme of this paper.

In the rest of the introduction, we will sketch some relevant background which is useful for the reader to understand the main results of this paper. We refer the reader to \cite{C08b}, \cite{C08}, \cite{E06} and the body of the paper for more details.
\subsection*{Trianguline representations and $p$-adic local Langlands correspondence for $\G$}
As usual, let $\r_L$ denote the Robba ring over $L$. The $\varphi$ and $\Gamma$-actions on $\r_L$ are defined by $\varphi(T)=(1+T)^p-1$ and $\gamma(T)=(1+T)^{\chi(\gamma)}-1$ for any $\gamma\in\Gamma$.  For any $\delta\in\mathrm{Hom}_{\mathrm{cont}}(\Q^\times,L^\times)$, we associate to $\delta$ a rank $1$ $\m$-module $\r_L(\delta)$ over $\r_L$ as follows: the $\m$-module $\r_L(\delta)$ has an $\r_L$-basis $e$ such that the $\varphi,\Gamma$-actions on $\r_L(\delta)$ are defined by the formulas
\[
\varphi(xe)=\delta(p)\varphi(x)e,\gamma(xe)=\delta(\chi(\gamma))\gamma(x)e
\]
for any $x\in\r_L$ and $\gamma\in\Gamma$, where $\chi$ is the cyclotomic character as usual. Conversely, if $M$ is a rank $1$ $\m$-module over $\r_L$, then there exists a unique $\delta\in\mathrm{Hom}_{\mathrm{cont}}(\Q^\times,L^\times)$ such that $M\cong\r_L(\delta)$ (\cite[Proposition 3.1]{C08b}). We define the \emph{weight} $w(\delta)$ of $\delta$  by the formula $w(\delta)=\log\delta(u)/\log u$, where $u\in\Z^\times$ is not a root of unity. The local reciprocity map allow us to view $\delta$ as a continuous character of $\mathrm{W}_{\Q}$.  If $\mathrm{val}(\delta(p))=0$, then one can uniquely extend $\delta$ to a continuous character of $G_{\Qp}$. In this case, $w(\delta)$ is just the generalized Hodge-Tate weight of $\delta$ and $\r_L(\delta)=\D_\rig^\dagger(\delta)$.

Recall that a $\m$-module over $\r_L$ is called \emph{triangulable} if it is a successive extensions of rank $1$ $\m$-modules over $\r_L$, and an $L$-linear representation $V$ of $G_{\Q}$ is called \emph{trianguline} if $\D_\rig^\dagger(V)$ is triangulable in the category of $\m$-modules over $\r_L$. In the rest of the introduction, let $V$ be a $2$-dimensional $L$-linear representation of $G_{\Q}$. If $V$ is trianguline, then $\D_\rig^\dagger(V)$ fits into a short exact sequence
\[
0\longrightarrow\r_L(\delta_1)\longrightarrow\D_\rig^\dagger(V)\longrightarrow\r_L(\delta_2)\longrightarrow0
\]
for some $\delta_1,\delta_2\in\mathrm{Hom}_{\mathrm{cont}}(\Q^\times,L^\times)$. We denote by $h\in H^1(\r_L(\delta_1\delta_2^{-1}))=\mathrm{Ext}^1(\r_L(\delta_2),\r_L(\delta_1))$ the extension corresponding to $\D_\rig^\dagger(V)$; then $V$ is determined by the triple $(\delta_1,\delta_2,h)$. It follows that $\mathrm{val}(\delta_1(p))+\mathrm{val}(\delta_2(p))=0$, and $w(\delta_1), w(\delta_2)$ are the generalized Hodge-Tate weights of $V$. Conversely, for any triple $s=(\delta_1,\delta_2,h)$ such that $\delta_1,\delta_2\in\mathrm{Hom}_{\mathrm{cont}}(\Q^\times,L^\times)$ and
$h\in H^1(\r_L(\delta_1\delta_2^{-1}))$, we denote by $D(s)$ the extension of $\r_L(\delta_2)$ by $\r_L(\delta_1)$ defined by $h$. If $\alpha\in L^\times$ and if $s'=(\delta_1,\delta_2,\alpha h)$, then $D(s)$ and $D(s')$ are isomorphic. Thus if $h\neq0$, then the isomorphism class of $D(s)$ only depend on the image $\bar{h}$ of $h$ in $\mathbf{P}^1(H^1(\r_L(\delta_1\delta_2^{-1})))$. Following the notations of Colmez, we denote by $\mathscr{S}_{+}(L)$ the set of triples $s=(\delta_1,\delta_2,\bar{h})$ where $\delta_1,\delta_2\in\mathrm{Hom}_{\mathrm{cont}}(\Q^\times,L^\times)$ such that $\mathrm{val}(\delta_1(p))+\mathrm{val}(\delta_2(p))=0$ and $\mathrm{val}(\delta_1(p))\geq0$, and
$\bar{h}\in\mathbf{P}^1(H^1(\r_L(\delta_1\delta_2^{-1})))$; then $D(s)$ is well-defined for any $s\in\mathscr{S}_+(L)$. In the case when $D(s)$ is \'etale, we denote by $V(s)$ the $L$-linear representation of $G_{\Q}$ such that $\D_\rig^\dagger(V(s))=D(s)$.

For any $s\in\mathscr{S}_+(L)$, we set
\[
u(s)=\mathrm{val}(\delta_1(p))=-\mathrm{val}(\delta_2(p)), w(s)=w(\delta_1)-w(\delta_2).
\]
In \cite{C08b}, Colmez defined three subsets $\mathscr{S}^{\mathrm{ng}}_{*}$, $\mathscr{S}^{\mathrm{cris}}_{*}$ and $\mathscr{S}^{\mathrm{st}}_{*}$ of $\mathscr{S}_+(L)$ as follows:
\begin{enumerate}
\item[(1)]$\mathscr{S}^{\mathrm{ng}}_{*}$ is the set of $s$ such that $w(s)$ is not an integer $\geq1$ and $u(s)>0$;
\item[(2)]$\mathscr{S}^{\mathrm{cris}}_{*}$ is the set of $s$ such that $w(s)$ is an integer $\geq1$, $0<u(s)< w(s)$ and $\bar{h}=\infty$;
\item[(3)]$\mathscr{S}^{\mathrm{st}}_{*}$ is the set of $s$ such that $w(s)$ is an integer $\geq1$, $0<u(s)< w(s)$ and $\bar{h}\neq\infty$.
\end{enumerate}
Here the exponents $``\mathrm{ng}",``\mathrm{cris}"$ and $``\mathrm{st}"$ refer to ``non-geometric'', ``crystalline'' and ``semistable'' respectively. Let
\[
\mathscr{S}_{\mathrm{irr}}=\mathscr{S}^{\mathrm{ng}}_{*}\coprod\mathscr{S}^{\mathrm{cris}}_{*}
\coprod\mathscr{S}^{\mathrm{st}}_{*}.
\]
It was proved by Colmez that if $s\in\mathscr{S}_{\mathrm{irr}}$, then $D(s)$ is \'etale, and $V(s)$ is irreducible (and of course trianguline); conversely, if $V$ is irreducible and trianguline, then $V\cong V(s)$ for some $s\in\mathscr{S}_{\mathrm{irr}}$ (\cite[Th\'eor\`{e}me 0.5(i)(ii)]{C08b}). For any $?\in\{\mathrm{ng},\mathrm{cris},\mathrm{st}\}$, we say $V\in\mathscr{S}^{?}_{*}$ if $V\cong V(s)$ for some $s\in\mathscr{S}^{?}_{*}$. (By \cite[Th\'eor\`{e}me 0.5(iii)]{C08b}, we know $V\in\mathscr{S}^{?}_{*}$ for at most one $?$.) More precisely, we have that $V\in\mathscr{S}^{\mathrm{cris}}_{*}$ if and only if  $V$ is a twist of an irreducible and crystabelian representation, and $V\in\mathscr{S}^{\mathrm{st}}_{*}$ if and only if $V$ is a twist of an irreducible and semistable, but non-crystalline, representation (\cite[Th\'eor\`{e}me 0.8]{C08b}).

In \cite{C04}, Colmez found a direct link between $\mathrm{B}(V)$ and the $\m$-module associated to $V$ in the semistable case.  More precisely, Colmez showed that if $V\in\mathscr{S}^{\mathrm{st}}_{*}$, then the following is true:
\begin{equation}\label{eq:prin-series}
\begin{split}
\text{the dual of}\hspace{1mm}\mathrm{B}(V) \hspace{1mm}\text{is naturally isomorphic to} \hspace{1mm}(\displaystyle{\lim_{\longleftarrow}}_{\psi}\D(\check{V}))^{b}\hspace{1mm}
\text{as Banach space representations of} \hspace{1mm}\mathrm{B}(\Q).
\end{split}
\end{equation}
Subsequently, Berger and Breuil proved $(\ref{eq:prin-series})$ for those $V\in\mathscr{S}^{\mathrm{cris}}_{*}$ which are not exceptional (\cite{BB06}) and Paskunas proved $(\ref{eq:prin-series})$ for $V$ exceptional and $p>2$ (\cite{P09}); for $V\in\mathscr{S}^{\mathrm{cris}}_{*}$, we call $V$ exceptional if the associated Weil-Deligne representation of $V$ is not Frobenius semi-simple, and Colmez proved $(\ref{eq:prin-series})$ for $V\in\mathscr{S}^{\mathrm{ng}}_{*}$ (\cite{C08c}). The isomorphism $(\ref{eq:prin-series})$ suggests a functorial construction of the $p$-adic local Langlands correspondence for $\G$ by using the theory of $\m$-modules. On this track, Colmez recently established the $p$-adic local Langlands correspondence for $\G$ for all $2$-dimensional irreducible $L$-linear representations of $G_{\Q}$ (\cite{C08}). To state Colmez's construction, let $D$ be a rank $2$, irreducible and \'etale $\m$-module over $\r_L$. In \cite{C08}, Colmez first equipped $D\boxtimes\mathbf{P}^1$ with a continuous $\G$-action. Then he showed that $D^\natural\boxtimes\mathbf{P}^1$ is stable under the given $\G$-action; to prove this assertion, Colmez improved $(\ref{eq:prin-series})$ to the following form:
\begin{equation}
\begin{split}
\text{the dual of}\hspace{1mm}\mathrm{B}(V) \hspace{1mm}\text{is naturally isomorphic to} \hspace{1mm}\D(\check{V})^\natural\boxtimes\mathbf{P}^1\hspace{1mm}
\text{as Banach space representations of} \hspace{1mm}\G
\end{split}
\end{equation}
when $V\in\mathscr{S}^{\mathrm{cris}}_*$ is not exceptional. Let
$\Pi(D)=(D\boxtimes\mathbf{P}^1)/(D^\natural\boxtimes\mathbf{P}^1)$; Colmez showed that the right hand side is an admissible unitary representation of $\G$. Colmez set the $p$-adic local Langlands correspondence for $\G$ as $V\mapsto\Pi(V):=\Pi(\D(V))$.

\subsection*{Locally analytic vectors of unitary principal series of $\G$}
In \cite{E06}, Emerton made the following conjecture ([Conjecture 3.3.1(8), ibid]):
\begin{conjecture}
For any $\eta, \psi\in\mathrm{Hom}_{\mathrm{cont}}(\Q^\times,L^\times)$, we have
\[
\dim \mathrm{Ref}^{\eta\otimes\psi}(V)=\dim\mathrm{Exp}^{\eta|x|\otimes x\psi }(\mathrm{B}(V)_{\mathrm{an}}\otimes(x|x|\circ\det)).
\]
\end{conjecture}
(Note that the right hand side of Conjecture 0.1 is $\dim\mathrm{Exp}^{\eta|x|\otimes x\psi }(\mathrm{B}(V)_{\mathrm{an}}\otimes(x|x|\circ\det))$ instead of $\dim\mathrm{Exp}^{\eta|x|\otimes x\psi}(\mathrm{B}(V)_{\mathrm{an}})$ in Emerton's formulation. This is because our normalization of the $p$-adic local Langlands correspondence for $\G$ differs by a twist of $(x|x|)^{-1}\circ \det$ from Emerton's normalization. See subsection 3.1 for more details.) Here $\mathrm{Ref}^{\eta\otimes\psi}(V)$ denotes the space of equivalence classes of refinements $[R]$ of $V$ such that $\sigma(R)=(\eta,\psi)$; for a locally analytic $\G$-representations $W$ of compact type, we denote by $\mathrm{Exp}^{\eta\otimes \psi}(W)$ the space of $1$-dimensional $\mathrm{T}(\Q)$-invariant subspaces of the Jacquet modules $J_{\mathrm{B}(\Q)}(W)$ on which $\mathrm{T}(\Q)$ acts via the character $\eta\otimes \psi$. Granting Emerton's conjecture, we see that $J_{\mathrm{B}(\Q)}(\mathrm{B}(V)_{\mathrm{an}})\neq0$ if and only if $V$ has at least one refinement; this is equivalent to the fact that $V$ is trianguline. Thus, inspired by the classical theory of smooth representations of $\mathrm{GL}_2$, it is reasonable to think of $\mathrm{B}(V)$, when $V$ is trianguline, as a \emph{``unitary principal series representation of $\G$''}. So far as we know, this point of view has not yet been accepted as a formal definition; but it has been adopted in some literature (e.g. \cite{C08c}). We follow this point of view in this paper.

The motivation of this paper is to have an explicit description of $\mathrm{B}(V)_{\mathrm{an}}$ for $V\in\mathscr{S}^{\mathrm{cris}}_*$. By the classification of the representations $V\in\mathscr{S}^{\mathrm{cris}}_*$ mentioned in $0.1$, it suffices to figure out $\mathrm{B}(V)_{\mathrm{an}}$ when $V$ is an irreducible and crystabelian representation of Hodge-Tate weights $(0,k-1)$ for some integer $k\geq2$. By a result of Colmez (\cite[Proposition 4.14]{C08b}), such a $V$ is uniquely determined by a pair of smooth characters $(\alpha,\beta)$ of $\Q^\times$. Furthermore, Berger and Breuil showed that $\mathrm{B}(V)\cong B(\alpha)/L(\alpha)$, where $B(\alpha)=(\Ind^{\G}_{\mathrm{B}(\Q)}\alpha\otimes x^{k-2}\beta|x|^{-1})^{\mathcal{C}^{-\val(\alpha(p))}}$ and $L(\alpha)$ is a certain closed subspace of $B(\alpha)$ (\cite{BB06}). We denote by  $\pi(\alpha)$ the locally algebraic representation $(\Ind^{\G}_{\mathrm{B}(\Q)}\alpha\otimes x^{k-2}\beta|x|^{-1})^{\mathrm{lalg}}$ and $A(\alpha)$ the locally analytic principal series $(\Ind^{\G}_{\mathrm{B}(\Q)}\alpha\otimes x^{k-2}\beta|x|^{-1})^{\mathrm{an}}$; we set $\pi(\beta)$ and $A(\beta)$ by replacing $\alpha$ with $\beta$. Breuil constructed a natural continuous $\G$-equivariant map from $A(\alpha)\oplus_{\pi(\beta)}A(\beta)$ to $\mathrm{B}(V)_{\mathrm{an}}$, and made the following conjecture:
\begin{conjecture}(\cite[Conjecture 5.3.7, Conjecture 4.4.1]{BB06})
If $\alpha\neq\beta$, then the natural map $A(\alpha)\oplus_{\pi(\beta)}A(\beta)\ra\mathrm{B}(V)_{\mathrm{an}}$ is a topological isomorphism.
\end{conjecture}
The main result of this paper is the following theorem.
\begin{theorem}(Theorem 4.1)
Conjecture $0.2$ is true.
\end{theorem}
Our proof of Theorem 0.3 largely relies on Colmez's identification of the locally analytic vectors of $\Pi(V)$. In fact, Colmez showed that if $D$ is a rank 2, irreducible and \'etale $\m$-module over $\r_L$, then $(\check{\Pi}(D)_{\mathrm{an}})^\ast=D^\natural_{\rig}\boxtimes\mathbf{P}^1$ (\cite{C08}). To apply his result, we will construct a $\G$-equivariant commutative diagram
\begin{equation}
\xymatrix{
(B(\alpha)/L(\alpha))^\ast\ar[d] \ar[r]& \D^\natural(\check{V}_{\alpha,\beta})\boxtimes\mathbf{P}^1\ar[d] \\
 (A(\alpha)\oplus_{\pi(\beta)}A(\beta))^\ast \ar[r]&\D^\natural_\rig(\check{V}_{\alpha,\beta})\boxtimes\mathbf{P}^1,}
\end{equation}
where the upper vertical line is the natural isomorphism $(0.2)$. Then Theorem 0.3 follows easily from $(0.3)$. As an application of Theorem 0.3, we finally prove Conjecture 0.1 for those $V\in\mathscr{S}^{\mathrm{cris}}_*$ which are not exceptional.
\begin{corollary}(Corollary 5.7)
Conjecture $0.1$ is true when $V\in\mathscr{S}^{\mathrm{cris}}_*$ is not exceptional.
\end{corollary}

\section*{Notation and conventions}
Throughout this paper, we fix a finite extension $L$ over $\Q$ to be the coefficient field. Let $\val$ denote the $p$-adic valuation on $\overline{\mathbb{Q}}_p$, normalized by $\val(p)=1$; the corresponding norm is denoted by $|x|$. Let $\alpha,\beta$ denote a pair of smooth characters $\alpha,\beta:\Q^\times\ra L^\times$ such that $-(k-1)<\val(\alpha(p))\leq\val(\beta(p))<0$ and $\val(\alpha(p))+\val(\beta(p))=k-1$ for an integer $k\geq2$. Let $\alpha_p,\beta_p$ denote $\alpha(p)^{-1},\beta(p)^{-1}$ respectively. For any smooth character $\tau:\Z^\times\ra \OO_L^\times$, we let $n(\tau)$ denote the conductor of $\tau$. If $n(\tau)=0$, then we say $\tau$ is unramified. Otherwise, we say $\tau$ is ramified.

As usual, let $\chi$ denote the cyclotomic character. For any $m\geq0$, let $\mu_{p^m}$ denote the set of $p^m$-th roots of unity in $\overline{\mathbb{Q}}_p$, and we use $\eta_{p^m}$ to denote a primitive $p^m$-th root of unity. Following Fontaine's notations of $p$-adic Hodge theory, we suppose $\varepsilon=[(\eps^{(m)})_{m\geq0}]\in W(R)$, where $(\eps^{(m)})_{m\geq0}$ is a compatible sequence of primitive $p^m$-th roots of unity such that $(\eps^{(m+1)})^p=\eps^{(m)}$. For any $y\in\Q$, if $y\in p^{-m}\Z$ for some $m\in\mathbb{Z}$, then we set $e^{2\pi iy}=(\eps^{(m)})^{p^my}$, which is independent of the choice of $m$. Put $F_{m}=\Q(\mu_{p^m})$ and $L_m=L\otimes_{\Q}F_{m}$. Let $F_{\infty}=\cup_{m\geq0} F_{m}$, and let $\Gamma=\mathrm{Gal}(F_\infty/\Q)$. The Galois group $\Gamma$ is isomorphic to $\Z^\times$ via the $p$-adic cyclotomic character $\chi$. For any $m\geq1$ and $p>2$ (resp. $p=2$), we set $\Gamma_m=\chi^{-1}(1+p^m\Z)$ (resp. $\Gamma_m=\chi^{-1}(1+p^{m+1}\Z)$).
If $\tau:\Gamma\ra\OO_L^\times$ is a smooth character and if $n(\tau)=m$, then for any $\eta_{p^m}$, we define
\[
G(\tau,\eta_{p^m})=\sum_{\gamma\in\Gamma/\Gamma_m}\tau^{-1}(\gamma)\gamma(\eta_{p^m})\in L_m.
\]
We set $G(\tau)=G(\tau,\eps^{(m)})$.

Let $\mathrm{W}_{\Q}$ denote the Weil group of $\Q$. The local Artin map induces a topological isomorphism $\Q^\times\cong\mathrm{W}_{\Q}^\mathrm{ab}$, which we normalize by identifying $p$ with a lift of $\mathrm{Frob}^{-1}_p$ (i.e. geometric Frobenius). This allow us to identify the set of characters of $\Q^\times$ with the set of characters of $\mathrm{W}_{\Q}^\mathrm{ab}$. For any integer $n$, we write $x^n$ to denote the character defined by $x\mapsto x^n$. If $c\in L^\times$, we let $\mathrm{ur}(c):\Q^\times\ra L^\times$ denote the character that maps $p$ to $c$ and is trivial on $\Z^\times$. If we regard $\chi$ as a character of $\Q^\times$ via the local Artin map, then it is equal to $x|x|$.

Let $\mathrm{B}$ denote the subgroup of upper triangular matrices of $\mathrm{GL}_2$, and let $\mathrm{T}$ denote the subgroup of diagonal matrices of $\mathrm{GL}_2$.
\section{Irreducible crystabelian representations of $\G$}
In this section, we will study some locally algebraic representations $\pi(\alpha)$ (resp. $\pi(\beta)$) of $\G$ and their universal unitary completions $B(\alpha)/L(\alpha)$ (resp. $B(\beta)/L(\beta)$). These representations were first introduced by Breuil in the context of his $p$-adic local Langlands program of $\G$. The terminology ``irreducible crystabelian representations of $\G$'' refers to the unitary admissible representations of $\G$ which correspond to $2$-dimensional irreducible crystabelian representations of $G_{\Q}$ via the $p$-adic local Langlands correspondence for $\G$. In fact, as will be explained in subsection 3.1, $B(\alpha)/L(\alpha)$ is the unitary admissible representation assigned to certain $2$-dimensional irreducible crystabelian representation $V_{\alpha,\beta}$ of $G_{\Q}$ by the $p$-adic local Langlands correspondence for $\G$. Hence $B(\alpha)/L(\alpha)$ are examples of the irreducible crystabelian representations of $\G$. Furthermore, we will see in subsection 2.1 that the set of $2$-dimensional irreducible crystabelian representations of $G_{\Q}$ consists of the representations $V_{\alpha,\beta}(n)$ for all the pairs $(\alpha,\beta)$ and $n\in\mathbb{Z}$. It follows that the set of irreducible crystabelian representation of $\G$ consists of the representations $B(\alpha)/L(\alpha)\otimes(x|x|\circ\det)^n$ for all $\alpha$ and $n\in\mathbb{Z}$. This fact explains the title of this section.
\subsection{Some locally algebraic representations of $\G$}
Following the notations of \cite{BB06}, we define the locally algebraic representations $\pi(\alpha)$ and $\pi(\beta)$ as
\[
\pi(\alpha)=(\Ind^{\G}_{\mathrm{B}(\Q)}\alpha\otimes x^{k-2}\beta|x|^{-1})^{\mathrm{lalg}}\cong\mathrm{Sym}^{k-2}L^2\otimes_L(\Ind^{\G}_{\mathrm{B}(\Q)}\alpha\otimes\beta|x|^{-1})^{\mathrm{sm}}
\]
\[
\pi(\beta)=(\Ind^{\G}_{\mathrm{B}(\Q)}\beta\otimes x^{k-2}\alpha|x|^{-1})^{\mathrm{lalg}}\cong\mathrm{Sym}^{k-2}L^2\otimes_L(\Ind^{\G}_{\mathrm{B}(\Q)}\beta\otimes\alpha|x|^{-1})^{\mathrm{sm}}.
\]
We equip $\pi(\alpha)$ (resp. $\pi(\beta)$) with the unique locally convex topology such that the open sets are the lattices (a lattice of an $L$-vector space $V$ is an $\OO_L$-submodule which generates $V$ over $L$) of $\pi(\alpha)$ (resp. $\pi(\beta)$).

For any $F\in\pi(\alpha)$, we put
$f(z)=F(\left(
\begin{smallmatrix}
 0&1 \\
 -1&z
\end{smallmatrix}
\right))$. The map $F\mapsto f$ identifies $\pi(\alpha)$ with the set of functions $f:\Q\ra L$ which are locally polynomials with coefficients in $L$ and degree $\leq k-2$ such that $\beta\alpha^{-1}(z)|z|^{-1}f(1/z)|_{\Z-\{0\}}$ extend to elements of $\mathrm{Pol}^{k-2}(\Z,L)$ (the set of functions $f:\Z\ra L$ which are locally polynomials with coefficients in $L$ and degree $\leq k-2$). Under this identification, the action of $\G$ is given by the formula
\begin{equation}
\left(
\begin{matrix}
 a&b \\
 c&d
\end{matrix}
\right)\cdot f(z)=\alpha(ad-bc)\beta\alpha^{-1}(-cz+a)|-cz+a|^{-1}(-cz+a)^{k-2}f(\frac{dz-b}{-cz+a}).
\end{equation}
Exchanging $\alpha$ and $\beta$, we get the similar description of $\pi(\beta)$.

\subsection{Unitary completions}
To introduce a few general definitions concerning $L$-Banach space representations, let $K$ be an intermediate field of $L/\Q$, and let $G$ be a locally $K$-analytic group such as the $K$-points of an algebraic group (in this paper, $K=\Q$ and $G=\G$).
\begin{definition}
An \emph{$L$-Banach space representation} $U$ of $G$ is an $L$-Banach space $U$ together with an action of $G$ such that $G\times U\ra U$ is continuous. An $L$-Banach space representation $U$ is called \emph{unitary} if the topology of $U$ may be defined by a $G$-invariant norm.
\end{definition}
\begin{definition}
Let $V$ be a locally convex topological $L$-vector space equipped with a continuous $G$-action, and let $U$ be a unitary $L$-Banach space representation of $G$. We say that a given continuous $L$-linear $G$-equivariant map $V\ra U$ realizes $U$ as a \emph{universal unitary completion} of $V$ if any continuous $L$-linear $G$-equivariant map $V\ra W$, where $W$ is a unitary $L$-Banach space representation of $G$, factors uniquely through the given map $V\ra U$.
\end{definition}
The following is devoted to the constructions of the universal unitary completions of $\pi(\alpha)$ and $\pi(\beta)$, which is due to Berger and Breuil. See \cite{BB06} for more details. Let
\[
B(\alpha)=(\Ind^{\G}_{\mathrm{B}(\Q)}\alpha\otimes x^{k-2}\beta|x|^{-1})^{\mathcal{C}^{\val(\alpha_p)}}
\]
and
\[
B(\beta)=(\Ind^{\G}_{\mathrm{B}(\Q)}\beta\otimes x^{k-2}\alpha|x|^{-1})^{\mathcal{C}^{\val(\beta_p)}}.
\]
For any $F\in B(\alpha)$, set
$f(z)=F(\left(
\begin{smallmatrix}
 0&1 \\
 -1&z
\end{smallmatrix}
\right))$.
In this way, we identify $B(\alpha)$ with the $L$-vector space of functions $f:\Q\ra L$ satisfying the following two conditions:
\begin{enumerate}
\item[(1)]$f|_{\Z}$ is a $\mathcal{C}^{\val(\alpha_p)}$-function;
\item[(2)]$(\beta\alpha^{-1})(z)^{-1}|z|z^{k-2}f(1/z)|_{\Z-\{0\}}$ can be extended to a $\mathcal{C}^{\val(\alpha_p)}$ function on $\Z$.
\end{enumerate}
The action of $\G$ is given by the same formula as in $(1.1)$. By this identification, we can write
\[
B(\alpha)\cong\mathcal{C}^{\val(\alpha_p)}(\Z,L)\oplus\mathcal{C}^{\val(\alpha_p)}(\Z,L), f\mapsto f_1\oplus f_2,
\]
where $f_1(z)=f(pz)|_{\Z}$ and $f_2(z)$ is the extension of $(\beta\alpha^{-1})(z)^{-1}z^{k-2}f(1/z)|_{\Z-\{0\}}$ to $\Z$. The resulting $L$-Banach space structure of $B(\alpha)$ is defined by the norm
\[
\|f\|=\max(\|f_1\|_{\val(\alpha_p)},\|f_2\|_{\val(\alpha_p)}).
\]
It is not difficult to show that $\G$ acts on $B(\alpha)$ by continuous automorphisms with respect to this norm (\cite[lemme 4.2.1]{BB06}), and the natural $\G$-equivariant inclusion $\pi_{\alpha}\hookrightarrow B(\alpha)$ is continuous. Let $L(\alpha)\subset B(\alpha)$ denote the closure of the $L$-subspace generated by $z^j$ and $(\beta\alpha^{-1})(z-a)|z-a|^{-1}(z-a)^{k-2-j}$
for all $a\in\Q$ and integers $j$ that $0\leq j<\val(\alpha_p)$ (the fact that $z^j$ and $(\beta\alpha^{-1})(z-a)|z-a|^{-1}(z-a)^{k-2-j}$ are contained in $B(\alpha)$ is proved in \cite[Lemme 4.2.2]{BB06}). It is stable under the action of $\G$ by \cite[lemme 4.2.3]{BB06}.

Exchanging $\alpha$ and $\beta$, we get the similar description of $B(\beta)$, and we set $L(\beta)$ as the closure of the $L$-subspace generated by $z^j$ and $(\alpha\beta^{-1})(z-a)|z-a|^{-1}(z-a)^{k-2-j}$ for all $a\in\Q$ and integers $j$ that $0\leq j<\val(\beta_p)$.
\begin{proposition}(\cite[Th\'eor\`{e}me 4.3.1]{BB06})
The continuous $\G$-equivariant map $\pi(\alpha)\ra B(\alpha)/L(\alpha)$ realizes $B(\alpha)/L(\alpha)$ as the universal unitary completion of $\pi(\alpha)$. The same result holds if we replace $\alpha$ by $\beta$.
\end{proposition}
\subsection{Intertwining operators}
Recall that there exists, up to multiplication of a nonzero scalar, a unique nonzero $\G$-equivariant morphism
\[
I^{\mathrm{sm}}:(\Ind^{\G}_{\mathrm{B}(\Q)}\beta\otimes\alpha|x|^{-1})^{\mathrm{sm}}\ra
(\Ind^{\G}_{\mathrm{B}(\Q)}\alpha\otimes\beta|x|^{-1})^{\mathrm{sm}}
\]
defined by (in terms of locally constant functions on $\Q$)
\begin{equation}
\begin{split}
I^{\mathrm{sm}}(h)(z)=\int_{\Q}(\beta\alpha^{-1})(x-z)|x-z|^{-1}h(x)dx,
\end{split}
\end{equation}
where $dx$ is the Haar measure on $\Q$. Tensoring with the identity map on $\mathrm{Sym}^{k-2}L^2$, we get a nonzero $\G$-equivariant morphism $I:\pi(\beta)\ra\pi(\alpha)$. It is well-known that $I^{\mathrm{sm}}$ is a nontrivial isomorphism if $\alpha\neq\beta,\beta|x|$, and is the
identity if $\alpha=\beta$ (see \cite{Bum98}).

\begin{proposition}
We have the following commutative $\G$-equivariant diagram
\[
\xymatrix{
\pi(\beta)\ar[d] \ar^{I}[r]& \pi(\alpha)\ar[d] \\
 B(\beta)/L(\beta) \ar^{\widehat{I}}[r]&B(\alpha)/L(\alpha),}
\]
where $\widehat{I}$ is the continuous $\G$-morphism induced from $I$. In case $\alpha\neq\beta|x|$, $I$ and $\widehat{I}$ are isomorphisms.
\end{proposition}
\begin{proof}
This follows from the functoriality of universal unitary completions and the fact that $I$ is an isomorphism in case $\alpha\neq\beta|x|$.
\end{proof}

Now suppose $\alpha=\beta|x|$; in particular, $\val(\alpha_p)=\frac{k-2}{2}$. The operator $I^{\mathrm{sm}}$ induces the following two exact sequences of $\G$-representations
\begin{equation}
0\longrightarrow \beta\circ\det\longrightarrow (\Ind^{\G}_{\mathrm{B}(\Q)}\beta\otimes\alpha|x|^{-1})^{\mathrm{sm}}\stackrel{I^{\mathrm{sm}}}\longrightarrow (\beta\circ\det)\otimes_L\mathrm{St}\longrightarrow 0
\end{equation}
and
\begin{equation}
0\longrightarrow (\beta\circ\det)\otimes_L\mathrm{St}\longrightarrow
(\Ind^{\G}_{\mathrm{B}(\Q)}\alpha\otimes\beta|x|^{-1})^{\mathrm{sm}}\longrightarrow \beta\circ\det\longrightarrow 0,
\end{equation}
where $\mathrm{St}=(\Ind^{\G}_{\mathrm{B}(\Q)}1)^{\mathrm{sm}}/1$ is the Steinberg representation of $\G$. Thus $I$ induces the following two exact sequences of $\G$-representations
\begin{equation}
0\longrightarrow(\beta\circ\det)\otimes_L\mathrm{Sym}^{k-2}L^2\longrightarrow \pi(\beta)\stackrel{I}\longrightarrow ((\beta\circ\det)\otimes_L\mathrm{Sym}^{k-2}L^2)\otimes_L\mathrm{St}\longrightarrow 0
\end{equation}
and
\begin{equation}
0\longrightarrow  ((\beta\circ\det)\otimes_L\mathrm{Sym}^{k-2}L^2)\otimes_L\mathrm{St}\longrightarrow\pi(\alpha)\longrightarrow (\beta\circ\det)\otimes_L\mathrm{Sym}^{k-2}L^2\longrightarrow 0.
\end{equation}
For $\widehat{I}$, let $K(\beta)\subset B(\beta)$ be the closure of the $L$-subspace generated by $L(\beta)$ and the functions $f:\Q\ra L$ of the form
\begin{equation}
f(z)=\sum_{j\in J}\lambda_j(z-z_j)^{n_j}\val(z-z_j),
\end{equation}
where $J$ is a finite set, $\lambda_j\in L$, $z_i\in\Q$, $n_j\in\{\lfloor\frac{k-2}{2}\rfloor+1,\dots,k-2\}$ and $\deg(\sum_{j\in J}\lambda_j(z-z_j)^{n_j})<\frac{k-2}{2}$ (by \cite[Lemma 5.4.1]{BB06} the functions of the form $(1.7)$ are contained in $B(\beta)$, so $K(\beta)$ is well-defined).
\begin{proposition}(\cite[Proposition 5.4.2]{BB06})
We have the $\G$-equivariant short exact sequence of Banach spaces
\[
0\longrightarrow K(\beta)/L(\beta)
\longrightarrow B(\beta)/L(\beta)\stackrel{\widehat{I}}\longrightarrow B(\alpha)/L(\alpha)\longrightarrow 0.
\]
Thus $\widehat{I}$ induces an isomorphism from $B(\beta)/K(\beta)$ to $B(\alpha)/L(\alpha)$.
\end{proposition}
In the rest of this section, we will compute $I^{\mathrm{sm}}(1_{p^{n}\Z}\cdot e^{2\pi ixy})$ for any $n\in\mathbb{Z}$ and $y\in\Q^{\times}$, which will be used later. To do the computation, we set $\(m(V))=\sup(n(\beta\alpha^{-1}),1)$, and
\[
C(\alpha_p,\beta_p)=\left\{
\begin{array}{l}
(\frac{\beta_p}{p\alpha_p})^{m(\alpha,\beta)},\text{if $\beta\alpha^{-1}$ is ramified;}\\
\frac{1-\beta_p/p\alpha_p}{1-\alpha_p/\beta_p}, \text{if $\beta\alpha^{-1}$ is unramified}.
\end{array}
\right.
\]
For the main results of this paper, we need the computation in the cases $n=0,1$ and $\val(y)\leq-\(m(V))-1$ only.
\begin{lemma}
For $n\in\mathbb{Z}$, we have
\[
I^{\mathrm{sm}}(1_{p^{n}\Z}\cdot e^{2\pi ixy})
=C(\alpha_p,\beta_p)(\frac{\beta_p}{\alpha_p})^{\val(y)}G(\beta^{-1}\alpha, e^{\frac{2\pi iy}{p^{\(m(V))+\val(y)}}})1_{p^{n}\Z}\cdot e^{2\pi izy},
\]
if $n+\val(y)\leq-\(m(V))$.
\end{lemma}
\begin{proof}
For $z\in p^n\Z$, we have
\begin{equation}
\begin{split}
I^{\mathrm{sm}}(1_{p^n\Z}\cdot e^{2\pi ixy})(z)&=\int_{p^n\Z}\beta\alpha^{-1}(x-z)|x-z|^{-1}e^{2\pi ixy}dx\\
&=e^{2\pi izy}\int_{p^n\Z}\beta\alpha^{-1}(x)|x|^{-1}e^{2\pi ixy}dx\\
&=e^{2\pi izy}\sum_{l=n}^{\infty}p^l\int_{p^l\Z^\times}(\beta\alpha^{-1})(x)e^{2\pi ixy}dx.
\end{split}
\end{equation}
If we let $S_{m}\subset\Z^{\times}$ be a system of representatives of $(\mathbb{Z}/p^{m}\mathbb{Z})^{\times}$ for any $m\geq1$, then we get
\begin{equation}
\begin{split}
\int_{p^l\Z^\times}(\beta\alpha^{-1})(x)e^{2\pi ixy}dx&=\sum_{a\in S_{\(m(V))}}\int_{p^la+p^{l+\(m(V))}\Z}(\beta\alpha^{-1})(p^{l}a)e^{2\pi ixy}dx\\
&=(\frac{\alpha_p}{\beta_p})^l\sum_{a\in S_{\(m(V))}}(\beta\alpha^{-1})(a)\int_{p^la+p^{l+\(m(V))}\Z}e^{2\pi ixy}dx\\
&=\left\{\begin{array}{ll}
       p^{-l-\(m(V))}(\frac{\alpha_p}{\beta_p})^l\sum_{a\in S_{\(m(V))}}(\beta\alpha^{-1})(a)e^{2\pi ip^lay},\\
     \text{if $l+\(m(V))\geq-\val(y)$};\\
         0,$ $\text{if $l+\(m(V))<-\val(y)$}.
      \end{array}
       \right.
\end{split}
\end{equation}
Since $n+\val(y)\leq-m(\alpha,\beta)$, it follows from $(1.8)$ that
\begin{equation}
\begin{split}
I^{\mathrm{sm}}(1_{p^n\Z}\cdot e^{2\pi ixy})(z)&=e^{2\pi izy}\sum_{l=-m(\alpha,\beta)-\val(y)}^{\infty} p^{-\(m(V))}(\frac{\alpha_p}{\beta_p})^l\sum_{a\in S_{\(m(V))}}(\beta\alpha^{-1})(a)e^{2\pi ip^lay}
\end{split}
\end{equation}
We treat the case when $\beta\alpha^{-1}$ is ramified firstly. If $l+\(m(V))\geq-\val(y)$, then we set $m=\max\{-l-\val(y),0\}<\(m(V))$, and we have
\begin{equation}
\begin{split}
\sum_{a\in S_{m(\alpha,\beta)}}(\beta\alpha^{-1})(a)e^{2\pi ip^lay}&=\sum_{b\in S_{m}}e^{2\pi ip^lby}(\sum_{a\in S_{\(m(V))},a\equiv b(\mathrm{mod}p^{m}\Z)}(\beta\alpha^{-1})(a))\\
&=\left\{
\begin{array}{ll}
G(\beta^{-1}\alpha, e^{\frac{2\pi iy}{p^{\(m(V))+\val(y)}}}),\\
\text{if $l+\(m(V))=-\val(y)$};\\
0,$ $\text{if $l+\(m(V))>-\val(y)$}.
\end{array}
\right.
\end{split}
\end{equation}
Hence by $(1.10)$, $I^{\mathrm{sm}}(1_{p^n\Z}\cdot e^{2\pi ixy})(z)$ is equal to
\[
p^{-\(m(V))}(\frac{\beta_p}{\alpha_p})^{\(m(V))+\val(y)}G(\beta^{-1}\alpha, e^{\frac{2\pi iy}{p^{m(V)+\val(y)}}})e^{2\pi izy}=C(\alpha_p,\beta_p)G(\beta^{-1}\alpha, e^{\frac{2\pi iy}{p^{m(V)+\val(y)}}})e^{2\pi izy}
\]
when $\beta\alpha^{-1}$ is ramified. If $\beta\alpha^{-1}$ is unramified, then we have
\[
 \sum_{a\in S_{m(\alpha,\beta)}}(\beta\alpha^{-1})(a)e^{2\pi ip^lay}=p-1
\]
if $l+1>-\val(y)$, and
\[
\sum_{a\in S_{m(\alpha,\beta)}}(\beta\alpha^{-1})(a)e^{2\pi ip^lay}=-1
\]
if $l+1=-\val(y)$. So in this case $I^{\mathrm{sm}}(1_{p^n\Z}\cdot e^{2\pi ixy})(z)$ is equal to
\begin{equation}
\begin{split}
(-\frac{1}{p}(\frac{\alpha_p}{\beta_p})^{-\val(y)-1}+\frac{p-1}{p}\sum_{l=-\val(y)}^{\infty}(\frac{\alpha_p}{\beta_p})^l)e^{2\pi izy}&=\frac{1-\beta_p/p\alpha_p}{1-\alpha_p/\beta_p}(\frac{\beta_p}{\alpha_p})^{\val(y)}e^{2\pi izy}\\
&=C(\alpha_p,\beta_p)G(\beta^{-1}\alpha, e^{\frac{2\pi iy}{p^{m(V)+\val(y)}}})e^{2\pi izy}
\end{split}
\end{equation}
since $G(\beta^{-1}\alpha, e^{\frac{2\pi iy}{p^{m(V)+\val(y)}}})=1$ when $\beta\alpha^{-1}$ is unramified.
\end{proof}
\begin{remark}
The above Lemma is singled out from the proof of \cite[Lemme 5.1.2]{BB06}.
\end{remark}
\subsection{Locally analytic representations}
In this subsection we collect some of the basic notions and facts concerning the theory of locally analytic representations of $p$-adic analytic groups, which will be used in the rest of this paper. In most of the cases, we follow the notations used by Schneider and Teitelbaum. For more details, we refer readers to their fundamental papers \cite{ST02}, \cite{ST02b} and \cite{ST03}.

Throughout this subsection we let $U$ denote an $L$-Banach space representation of $G$.

\begin{definition}
In case when $G$ is compact, an $L$-Banach space representation $U$ is called \emph{admissible} if there is a $G$-invariant bounded open $\OO_L$-submodule $M$ of $U$ such that, for any open normal subgroup $H$ of $G$, the $\OO_L$-module $(U/M)^H$ is of cofinite type. If $G$ is not compact, we call $U$ \emph{admissible} if it is admissible as a representation for one (or equivalently any) compact open subgroup of $G$.
\end{definition}
For compact $G$, the dual of the $L$-valued continuous functions on $G$ is isomorphic to $\Lambda[[G]]:=L\otimes_{\Z}\Z[[G]]$, the \emph{Iwasawa algebra of measures}. The $G$-action on $U$ extends naturally to an action of the algebra $\Lambda[[G]]$ by continuous linear endomorphisms on $U$. By functoriality, $\Lambda[[G]]$ also acts on the continuous dual $U^\ast$ of $U$. Then $U$ is admissible if and only if $U^\ast$ is finitely generated as a $\Lambda[[G]]$-module (\cite[Lemma 3.4]{ST02b}).
\begin{definition}
A \emph{locally analytic $G$-representation} $W$ over $L$ is a barrelled locally convex Hausdorff $L$-vector space $W$ equipped with a $G$-action by continuous linear endomorphisms such that, for each $v\in V$, the orbit map $g\mapsto g\cdot v$ is a $W$-valued locally analytic function on $G$.
\end{definition}

Let $A$ be an $L$-Fr\'echet algebra. For a continuous seminorm $q$ on $A$, it induces a norm on the quotient space $A/\{a\in A:q(a)=0\}$. Let $A_q$ denote the completion of the latter with respect to $q$. For any two continuous seminorms $q'\leq q$ the identity on $A$ extends to a continuous linear map $\phi^{q'}_q: A_q\ra A_{q'}$.
\begin{definition}
The $L$-Fr\'echet algebra $A$ is called an \emph{$L$-Fr\'echet-Stein algebra} if there is a sequence $q_1\leq\dots\leq q_n\leq\dots$ of continuous seminorms on $A$ which define the Fr\'echet topology such that for any $n\in \mathbb{N}$, we have
\begin{enumerate}
\item[(1)] $A_{q_n}$ is left noetherian;
\item[(2)] $A_{q_n}$ is flat as a right $A_{q_{n+1}}$-module via $\phi^{q_n}_{q_{n+1}}$.
\end{enumerate}

\end{definition}
We fix an $L$-Fr\'echet-Stein algebra $A$ and a sequence $(q_n)_{n\in\mathbb{N}}$ as in the above definition.
\begin{definition}
A \emph{coherent sheaf} for $(A,(q_n))$ is a family $(M_n)_{n\in\mathbb{N}}$ where each $M_n$ is a finitely generated $A_{q_n}$-module together with isomorphisms $A_{q_n}\otimes_{A_{q_{n+1}}}M_{n+1}\cong M_n$ as $A_{q_n}$-modules for any $n\in\mathbb{N}$. The global sections of $(M_n)_n$ is defined by
\[
\Gamma((M_n)_n):=\lim_{\longleftarrow n}M_n.
\]
\end{definition}

\begin{definition}
A left $A$-module $M$ is called \emph{coadmissible} if it is isomorphic to the module of global sections of some coherent sheaf $(M_n)_n$ for $(A,(q_n))$. Each $M_n$ carries its canonical Banach space topology as a finitely generated $A_{q_n}$-module. We equip $M$ with the projective limit topology which makes $M$ into an $L$-Fr\'echet space. We call this topology the \emph{canonical topology} of $M$.
\end{definition}
\begin{remark}
A simple cofinality argument shows that the canonical topology of a coadmissible module is independent of the choice of the sequence $(q_n)_n$.
\end{remark}
For compact $G$, let $D(G,L)$ denote the algebra of locally analytic distributions on $G$. This algebra is the continuous dual of the locally analytic $K$-valued functions on $G$, with multiplication given by convolution. For a locally analytic representation $W$ over $L$, the $G$ action extends naturally to an action of $D(G,L)$, yielding an action of $D(G,L)$ on $W^\ast$. The crucial property of $D(G,L)$ is that:

\begin{proposition}(\cite[Theorem 5.1]{ST03})
$D(G,L)$ is a Fr\'echet-Stein algebra.
\end{proposition}
\begin{definition}
In case $G$ is compact, an \emph{admissible locally analytic $G$-representation} over $L$ is a locally analytic $G$-representation on an $L$-vector space of compact type $W$ such that the strong dual $W_b'$ is a coadmissible $D(G,L)$-module equipped with its canonical topology. For general $G$, a locally analytic $G$-representation over $L$ is called \emph{admissible} if it is admissible as an $H$-representation for one (or equivalently any) open compact subgroup $H$ of $G$.
\end{definition}

\begin{definition}
 A vector $u\in U$ is called \emph{locally analytic} if the continuous orbit map $g\mapsto g\cdot u$ is a $U$-valued locally analytic function on $G$. We denote by $U_{\mathrm{an}}$ the $L$-vector subspace of locally analytic vectors of $U$, and we equip $U_{\mathrm{an}}$ with the subspace topology.
\end{definition}
Since the locally analytic functions are a subspace of the continuous functions, there is a natural morphism $\Lambda[[G]]\ra D(G,L)$.

\begin{proposition}
If $U$ is an admissible $L$-Banach space representation, then $U_{\mathrm{an}}$ is an admissible locally analytic $G$-representation and $(U_{\mathrm{an}})^\ast_b\cong D(H,L)\otimes_{\Lambda[[H]]}U^\ast$ for any open compact subgroup $H$ of $G$.
\end{proposition}
\begin{proof}
See \cite[Theorem 7.1]{ST03}.
\end{proof}
\begin{example}
The locally analytic principal series $A(\alpha)=(\Ind^{\G}_{\mathrm{B}(\Q)}\alpha\otimes x^{k-2}\beta|x|^{-1})^{\mathrm{an}}$ and $A(\beta)=(\Ind^{\G}_{\mathrm{B}(\Q)}\beta\otimes x^{k-2}\alpha|x|^{-1})^{\mathrm{an}}$ are admissible locally analytic representations.
As for $B(\alpha)$, for any $F\in A(\alpha)$, we associate $F$ with $f(z)=F(\left(
\begin{smallmatrix}
 0&1 \\
 -1&z
\end{smallmatrix}
\right))$. The map $F\mapsto f$ identifies $A(\alpha)$ with the $L$-vector space of functions $f:\Q\ra L$ satisfying the following two conditions:
\begin{enumerate}
\item[(1)]$f|_{\Z}$ is a locally analytic function;
\item[(2)]$(\beta\alpha^{-1})(z)^{-1}|z|z^{k-2}f(1/z)|_{\Z-\{0\}}$ extends to a locally analytic function on $\Z$.
\end{enumerate}
We make the similar identification of $A(\beta)$.
\end{example}
\section{Crystabelian representations of $G_{\Q}$}

This section is devoted to the study of (2-dimensional) crystabelian representations of $G_{\Q}$. To fix notation, recall that an $L$-linear (resp. $\OO_L$-) representation of $G_{\Q}$ is a finite dimensional $L$-vector space $V$ (resp. finite type $\OO_L$-modules $M$) equipped with a continuous linear action of $G_{\Q}$. Throughout this section, let $V$ be an $L$-linear representation of $G_{\Q}$, and let $M$ be a free $\OO_L$-representation.

\subsection{Classification of $2$-dimensional irreducible crystabelian representations of $G_{\Q}$}
\begin{definition}
We call an $L$-linear representation $V$ of $G_{\Q}$ \emph{crystabelian (crystalline abelian)} if there exists $n\geq0$ such that the restriction of $V$ to $G_{F_n}$ is crystalline, or in other words, $V$ becomes crystalline over an abelian extension of $\Q$. We then define $n(V)$ as the minimal integer $n\geq1$ such that the restriction of $V$ on $G_{F_n}$ is crystalline. We define $m(V)=\min_{\tau}n(V(\tau))$ where $\tau$ goes through all the finite order characters of $\Gamma$. We call $m(V)$ the \emph{essential conductor} of $V$.
\end{definition}
For $V$ crystabelian, we define
$\D_{\mathrm{cris}}(V)=\bigcup_{n=0}^\infty(\bcris\otimes_{\Q}V)^{G_{F_n}}=(\bcris\otimes_{\Q}V)^{G_{F_n}}$ for any $n\geq n(V)$, which is a weakly admissible filtered $(\varphi,G_{\Q})$-module over $L$.
If $F_{n(V)}\subset K\subset F_{\infty}$, then we have $K\otimes_{\Q}\D_{\rm{cris}}(V)=K\otimes_{\Q}\D_{\rm dR}(V)$.
Note that $G_{F_{n(V)}}$ acts trivially on $\D_{\mathrm{cris}}(V)$.

In the following, we will classify the set of $2$-dimensional irreducible crystabelian representations of $G_{\Q}$ with Hodge-Tate weights $(0,k-1)$ in terms of the weakly admissible $(\varphi,G_{\Q})$-modules $\D_{\mathrm{cris}}(V)$.
\begin{definition}
Let $D(\alpha,\beta)$ denote the filtered $(\varphi,G_{\Q})$-module over $L$ defined by $D(\alpha,\beta)=Le_\alpha\oplus Le_\beta$ and
\begin{enumerate}
\item[(i)]if $\alpha\neq\beta$, then $\varphi(e_\alpha)=\alpha(p)e_\alpha$, $\varphi(e_\beta)=\beta(p)e_\beta$ and $\gamma(e_\alpha)=\alpha(\chi(\gamma))e_\alpha$, $\gamma(e_\beta)=\beta(\chi(\gamma))e_\beta$ for $\gamma\in\Gamma$ and for $n\geq \max\{n(\alpha),n(\beta)\}$,
\[
\mathrm{Fil}^i(L_{n}\otimes_L D(\alpha,\beta))=\left\{
         \begin{array}{lll}
          L_n\otimes_LD(\alpha,\beta) & \text{if $i\leq-(k-1)$}; \\
          L_n\cdot(e_\alpha+G(\alpha\beta^{-1})e_\beta)  & \text{if $-(k-2)\leq i\leq0$}; \\
          0  & \text{if $i\geq1$}.
         \end{array}
       \right.
\]
\item[(ii)]if $\alpha=\beta$, then $\varphi(e_\alpha)=\alpha(p)e_\alpha$, $\varphi(e_\beta)=\beta(p)(e_\beta-e_\alpha)$ and $\gamma(e_\alpha)=\alpha(\chi(\gamma))e_\alpha$, $\gamma(e_\beta)=\beta(\chi(\gamma))e_\beta$ for $\gamma\in\Gamma$ and for $n\geq n(\alpha)$,
\[
\mathrm{Fil}^i(L_{n}\otimes_L D(\alpha,\beta))=\left\{
         \begin{array}{lll}
          L_n\otimes_LD(\alpha,\beta) & \text{if $i\leq-(k-1)$}; \\
          L_n\cdot e_\beta  & \text{if $-(k-2)\leq i\leq0$}; \\
          0  & \text{if $i\geq1$}.
         \end{array}
       \right.
\]
\end{enumerate}
\end{definition}

\begin{proposition}(\cite[Proposition 4.14]{C08b})
If $V$ is a 2-dimensional irreducible crystabelian representation of $G_{\Q}$ with Hodge-Tate weights $(0,k-1)$, then there exists a unique pair $(\alpha,\beta)$ such that $D(\alpha,\beta)=\D_{\rm cris}(V)$. Conversely, for any pair $(\alpha,\beta)$, there exists a unique two dimensional irreducible crystabelian representation $V$ of $G_{\Q}$ with Hodge-Tate weights $(0,k-1)$ such that $\D_{\rm cris}(V)=D(\alpha,\beta)$.
\end{proposition}

Henceforth we denote by $V_{\alpha,\beta}$ the crystabelian representation $V$ such that $\D_{\mathrm{cris}}(V)=D(\alpha,\beta)$. We have $n(V_{\alpha,\beta})=\max(n(\alpha),n(\beta))$ and $m(V_{\alpha,\beta})=\max(n(\alpha\beta^{-1}),1)=m(\alpha,\beta)$.

\begin{corollary}
If $V$ is a $2$-dimensional irreducible crystabelian representation of $G_{\Q}$, then there exists unique pair $(\alpha,\beta)$ and $n\in\mathbb{Z}$ such that $V$ is isomorphic to $V_{\alpha,\beta}(n)$.
\end{corollary}
\subsection{$\m$-modules}
In this subsection, we recall some of the basic theory of $\m$-modules of $p$-adic representations. The theory of $\m$-modules is the main ingredient of Colmez's construction of $p$-adic local Langlands correspondence for $\G$ as will be explained in next section. The notion of $\m$-modules will also be used in 2.3. For our purpose, we restrict to the case of $G_{\Q}$-representations. We refer the reader to the papers \cite{Fo91}, \cite{CC98}, \cite{LB02}, \cite{BB06} and \cite{C07} for more details.

We begin by recall some of the rings used in the theory of $\m$-modules.
\begin{enumerate}
\item[(i)]Let $\calE_L^+$ denote the ring $L\otimes_{\OO_L}\OO_L[[T]]$.
\item[(ii)]Let $\OO_{\calE_L}$ be the ring consisting of series $\sum_{i\in\mathbb{Z}}a_iT^i$ such that $a_i\in\OO_L$ and $a_{i}\ra0$ as $i\ra-\infty$. We equip $\OO_{\calE_L}$ with a valuation $w$ by setting $w(g(T))=\min_{i\in\mathbb{Z}}\val(a_i)$ if $g(T)=\sum_{i\in\mathbb{Z}}a_iT^i$. One can show that $\OO_{\calE_L}$ is a complete discrete valuation ring with respect to $w$. The fraction field of $\OO_{\calE_L}$ is $\calE_L=\OO_{\calE_L}[1/p]$; this is a local field of dimension $2$.
\item[(iii)]Let $\calE^{]0,r]}_L$ be the ring of formal series $g(T)=\sum_{i\in\mathbb{Z}}a_iT^i$ such that $g(T)$ is convergent on the annulus $r\geq\val(T)>0$. We define a norm $\|\cdot\|_{r}$ on $\calE^{]0,r]}_L$ by the formula
\[
\|g(T)\|_{r}=\sup_{i\in\mathbb{Z}}|a_i|p^{-ri}.
\]
Let $\r_L=\bigcup_{r>0}\calE^{]0,r]}_L$. In other words, $\r_L$ is the set of $p$-adic holomorphic functions on the boundary of the open unit disk. Let $\r^+_L=\r_L\cap L[[T]]$.

\item[(iv)]Let $\calE^{(0,r]}_L=\calE_L\cap\calE^{]0,r]}_L$. Then $\calE_L^{(0,r]}$ can be regarded as the subring of $\calE^{]0,r]}_L$ consisting of series with bounded coefficients. Let $\calE^\dagger_L=\bigcup_{r>0}\calE^{(0,r]}_L=\r_L\cap\calE_L$ and $\OO_{\calE^\dagger_L}=\r_L\cap\OO_{\calE_L}$. One can show that $\OO_{\calE^\dagger_L}$ is a discrete valuation ring with respect to $w$, and $\calE^\dagger_{L}$ is the fraction field of $\OO_{\calE^\dagger_L}$. The ring $\OO_{\calE_L}$ is the completion of $\OO_{\calE^\dagger_L}$ with respect to $w$.
\end{enumerate}

We equip $\OO_{\calE_L}$ with the weak topology by taking $\{\pi_L^i\OO_{\calE_L}+T^j\OO_L[[T]]\}_{i,j\geq0}$ as a basis of open neighborhoods of $0$. The weak topology on $\calE_L=\bigcup_{k\geq0}\pi_L^{-k}\OO_{\calE_L}$ is the inductive limit topology.
This topology induces the $(\pi_L,T)$-adic topology on $\calE_L^+$. We equip $\r^+_L$ with the Fr\'echet topology defined by the set of norms $\{\|\cdot\|_{r}\}_{r>0}$.

Let $R$ denote any of the rings $\calE_L^+,\OO_{\calE_L},\calE_L,\calE^\dagger_L,\r^+_L$ and $\r_L$. We equip the ring $R$ with commuting actions of $\varphi$ and $\Gamma$ by setting $\varphi(g(T))=g((1+T)^p-1)$ and $\gamma(g(T))=g((1+T)^{\chi(\gamma)}-1)$ for any $g(T)\in R$ and $\gamma\in\Gamma$. It is not difficult to see that $\Gamma$ acts on $R$ by isometries, and $\varphi$ is continuous. The ring $R$ is a finite free $\varphi(R)$-module of rank $p$ with a basis $\{(1+T)^i\}_{0\leq i\leq p-1}$. Thus for any $g\in R$, we can write $g$ in the form $g=\sum_{i=0}^{p-1}(1+T)^i\varphi(g_i)$ uniquely. We define the operator $\psi:R\ra R$ by the formula $\psi(g)=g_0$. Then it follows that $g_i=\psi((1+T)^{-i}g), \psi(\phi(g)h)=g\psi(h)$ for any $g,h\in R$, and $\psi$ commutes with $\Gamma$.

A \emph{$\varphi$-module} over $\oe$ is a finite type $\oe$-module $D$ equipped with a $\varphi$-semilinear $\oe$-morphism $\varphi:D\ra D$. We call $D$ \emph{\'etale} if the natural $\oe$-linear map $\varphi^\ast D=\oe\otimes_{\varphi,\oe}D\ra D$, sending $g\otimes x$ to $g\varphi(x)$ for $g\in\oe$ and $x\in D$, is an isomorphism. A \emph{$\varphi$-module} over $\calE_L$ is a finite dimensional $\calE_L$-vector space $D$ equipped with a $\varphi$-semilinear $\calE_L$-morphism $\varphi:D\ra D$. A $\varphi$-module $D$ over $\calE_L$ is called \emph{\'etale} if $D$ has an $\oe$-lattice which is $\varphi$-stable and \'etale. We define the notion of $\varphi$-modules over $\calE_L^\dagger$ and $\r_L$ similarly. If $D^\dagger$ is a $\varphi$-module over $\calE^\dagger_L$, then $D=\calE_L\otimes_{\calE^\dagger_L}D^\dagger$ is a $\varphi$-module over $\calE_L^\dagger$, and we call $D^\dagger$ \emph{\'etale} if $D$ is. A $\varphi$-module $D_\rig$ over $\r_L$ is called \emph{\'etale} if $D_\rig$ is pure of slope $0$ in the sense of Kedlaya \cite{Ke06}. We have the following result \cite[Proposition 1.5.5]{Ke06}.
\begin{theorem}
The functor $D^\dagger\mapsto\r_L\otimes_{\calE^\dagger_L}D^\dagger$, from the category of \'etale $\varphi$-modules over $\calE_L^\dagger$ to the category of \'etale $\varphi$-modules over $\r_L$, is an equivalence of categories.
\end{theorem}
For any $R$ of the rings $\OO_{\calE_L},\calE_L,\calE^\dagger_L$ and $\r_L$, a $\m$-module over $R$ is a $\varphi$-module $D$ over $R$ equipped with a continuous semilinear $\Gamma$-action which commutes with $\varphi$. We call $D$ \emph{\'etale} if $D$ is \'etale as a $\varphi$-module over $R$. If $D$ is an \'etale $\varphi$-module over $R$ and if $x\in D$, then we can write $x=\sum_{i=0}^{p-1}(1+T)^i\varphi(x_i)$ where $x_i\in D$ is uniquely determined for $0\leq i\leq p-1$. We define the operator $\psi:D\ra D$ by the formula $\psi(x)=x_0$. It follows that $x_i=\psi((1+T)^{-i}x), \psi(\phi(g)x)=g\psi(x), \psi(g(\varphi(x))=\psi(g)x$ for any $g\in R$ and $x\in D$. If $D$ is further an \'etale $\m$-module, then $\psi$ commutes with $\Gamma$.

If $D$ is an \'etale $\m$-module over $\calE_L$ (resp. $\oe$), then $\mathrm{V}(D)=(\widehat{\calE_L^{ur}}\otimes_{\calE_L} D)^{\varphi=1}$ (resp. $\mathrm{V}(D)=(\OO_{\widehat{\calE_L^{ur}}}\otimes_{\oe} D)^{\varphi=1}$) is an $L$-linear (resp. free $\OO_L$-) representation of $G_{\Q}$. One can show that $\dim_L(\mathrm{V}(D))=\dim_{\calE_L}D$ (resp. $\rank_{\OO_L}(\mathrm{V}(D))=\rank_{\OO_{\calE_L}}D$). We have the following result \cite[$\mathrm{A}3.4$]{Fo91}.

\begin{theorem}
The functor $D\mapsto\mathrm{V}(D)$, from the category of \'etale $(\varphi,\Gamma)$-modules over $\calE_L$ (resp. $\OO_{\calE_L}$) to the category of $L$-linear (resp. free $\OO_L$-) representations of $G_{\Q}$, is an equivalence of categories. The inverse functor is given by $\D(V)=(\widehat{\calE_L^{ur}}\otimes_L V)^{\mathrm{Gal}(\overline{\mathbb{Q}}_p/F_{\infty})}$ (resp. $\D(M)=(\OO_{\widehat{\calE_L^{ur}}}\otimes_{\OO_L} M)^{\overline{\mathbb{Q}}_p/F_{\infty})}$).
\end{theorem}
Let $\bb^{\dagger,r},\bb^\dagger$ and $\aa^\dagger$ be the rings constructed in \cite[1.3]{LB02}. Here $\bb^\dagger=\bigcup_{r>0}\bb^{\dagger,r}$ is a subfield of $\widehat{\calE_L^{ur}}$, and $\aa^\dagger$ is contained in $\bb^\dagger$. Both $\aa^\dagger$ and $\bb^\dagger$ are stable under the $\varphi,\Gamma$-actions. For any $r>0$, Let $\D^{\dagger,r}(V)=(\bb^{\dagger,r}\otimes_LV)^{\mathrm{Gal}(\overline{\mathbb{Q}}_p/F_{\infty})}$. Let $\D^\dagger(V)=\bigcup_{r>0}\D^{\dagger,r}(V)=
(\bb^{\dagger}\otimes_LV)^{\mathrm{Gal}(\overline{\mathbb{Q}}_p/F_{\infty})}$ and $\D^{\dagger}(M)=(\aa^{\dagger}\otimes_{\OO_L}M)^{\mathrm{Gal}(\overline{\mathbb{Q}}_p/F_{\infty})}$.
We have the following result \cite{CC98}.
\begin{theorem}
There exists an $r(V)$ such that
$\D(V)=\calE_L\otimes_{\calE^{(0,r]}_L}\D^{\dagger,r}(V)$ if $r\geqslant r(V) $.
Equivalently, $\D^{\dagger}(V)$ is an \'etale $\m$-module over
$\calE_L^\dagger$ with $\dim_{\calE^\dagger_L}(\D^\dagger(V))=\dim_LV$. As a consequence, the functor $\D^\dagger$, from the category of $L$-linear (free $\OO_L$-) representations of $G_{\Q}$ to the category of \'etale $\m$-modules over $\calE_L^\dagger$ (resp. $\OO_{\calE_L^\dagger}$), is an equivalence of categories. The inverse functor is given by $\mathrm{V}(D^\dagger)=(\widehat{\calE_L^{ur}}\otimes_{\calE_L^\dagger} D^\dagger)^{\varphi=1}$ (resp. $\mathrm{V}(D^\dagger)=(\OO_{\widehat{\calE_L^{ur}}}\otimes_{\OO_{\calE_L^\dagger}} D^\dagger)^{\varphi=1}$).
\end{theorem}
 Let $\D_\rig^{\dagger,r}(V)=\calE^{]0,r]}_L\otimes_{\calE^{(0,r]}_L}\D^{\dagger,r}(V)$ and $\D_\rig^\dagger(V)=\bigcup_{r>0}\D_\rig^{\dagger,r}(V)=\r_L\otimes_{\calE_L}\D^\dagger(V)$. Combining Theorem 2.7 and Theorem 2.5, we get the following result.
\begin{theorem}
We have that $\D_\rig^{\dagger,r}(V)$ is a free $\calE^{]0,r]}_L$-module with $\rank_{\calE^{]0,r]}_L}(\D_\rig^{\dagger,r}(V))=\dim_LV$ for $r$ sufficiently large. As a consequence, the functor $\D_\rig^\dagger$, from the category of $L$-linear representations of $G_{\Q}$ to the category of $\m$-modules over $\r_L$, is an equivalence of categories.
\end{theorem}

If $D$ is a finite type $\oe$-module of rank $d$, then we equip $D$ with the weak topology induced from the weak topology of $\oe$.
\begin{definition}
A \emph{treillis} of a finite type $\oe$-module $D$ is a compact $\OO_L[[T]]$-submodule $N$ of $D$ such that the image of $N$ in $D/\pi_LD$ is a $k_L[[T]]$-lattice. A treillis of a finite dimensional $\calE_L$-vector space $D$ is a treillis of an $\oe$-lattice of $D$.
\end{definition}
\begin{proposition}(\cite[Proposition 2.17]{C07})
If $D$ is an \'etale $\varphi$-module over $\oe$, then there exists a unique treillis $D^\sharp$ of $D$ satisfying the following properties:
\begin{enumerate}
\item[(i)]for every $x\in D$ and $i\in \mathbb{N}$, there exists $n(x,i)\in\mathbb{N}$ such that $\psi^n(x)\in D^\sharp+\mathfrak{m}_L^iD$ if $n\geq n(x,i)$;
\item[(ii)]$\psi(D^\sharp)=D^\sharp$.
\end{enumerate}
Moreover,
\begin{enumerate}
\item[(iii)]if $N$ is a treillis of $D$ and $i\in\mathbb{N}$, then there exists $n(N,i)$ such that $\psi^n(N)\subseteq D^\sharp+\mathfrak{m}_L^iD$ if and only if $n\geq n(N,i)$;
\item[(iv)]if $N$ is a treillis of $D$ stable under $\psi$ such that $\psi(N)=N$, then $TD^\sharp\subseteq N\subseteq D^\sharp$.
\end{enumerate}
\end{proposition}
\begin{proposition}(\cite[Corollaire 2.31]{C07})
If $D$ is an \'etale $\varphi$-module over $\oe$, then the set of $\psi$-stable treillis of $D$ admits a unique minimal element $D^\natural$, and $\psi(D^\natural)=D^\natural$.
\end{proposition}
We let $\D^\sharp(M)$ denote the treillis associated to $\D(M)$ by Proposition 2.10. If $V$ is an $L$-linear representation of $G_{\Q}$, we choose $M$ a $G_{\Q}$-invariant lattice of $V$, and put $\D^\sharp(V)=\D^\sharp(M)\otimes_{\OO_L}L$; it is independent of the choice of $M$. We define $\D^\natural(M)$, $\D^\natural(V)$ similarly.

\subsection{Wach modules of crystabelian representations of $G_{\Q}$}
In this subsection, we recall some of the basic theory of Wach modules of crystabelian representations of $G_{\Q}$ developed in \cite{BB06}. The notation of Wach modules is used to relate Berger-Breuil's and Colmez's constructions in the case of crystabelian representations as we will see in section 3.

Let $\bb^+=A_S^+[1/p]$ be the ring constructed in \cite[B1.8]{Fo91}. The ring $\bb^+$ is contained in $\widehat{\calE_L^{ur}}$ and stable under the $\varphi,\Gamma$-actions. We define $\D^+(V)=(\bb^+\otimes_{\Q}V)^{\mathrm{Gal}(\overline{\mathbb{Q}}_p/F_{\infty})}$, which is a finite type $\calE_L^+$-submodule of $\D(V)$. Recall that a Hodge-Tate representation is called positive if its Hodge-Tate weights are all $\leq0$. We have the following result \cite[Th\'eor\'eme 3.1.1]{BB06}.

\begin{theorem}
If $V$ is a positive crystabelian representation, then there exists a unique $\calE_L^+$-submodule $\mathrm{N}(V)$ of $\D^+(V)$ satisfying the following conditions:
\begin{enumerate}
\item[(i)]we have $\D(V)=\calE_L\otimes_{\calE_L^+}\mathrm{N}(V)$;
\item[(ii)]the $\Gamma$-action preserves $\mathrm{N}(V)$ and is finite on $\mathrm{N}(V)/T\mathrm{N}(V)$;
\item[(iii)]there exists $h\geq0$ such that $T^h\D^+(V)\subset \mathrm{N}(V)$.
\end{enumerate}
The module $\mathrm{N}(V)$ is also stable under the $\varphi$-action.
\end{theorem}

For any $m\geq1$, the map $\iota_m=\varphi^{-m}: \bb^+\ra\bdR^+$ induces a map $\iota_m:\D^+(V)\ra\bdR^+\otimes_{\Q}V$. We extend it to a map $\iota_m:\r^+_L[1/t]\otimes_{\calE_L^+}\D^+(V)\ra\bdR\otimes_{\Q}V$ by setting $\iota_m(T)=\eps^{(m)}e^{t/p^{m}}-1$.  Here $\iota_m$ is a special case of the \emph{localization map}. For the construction and general properties of the localization map, we refer the reader to \cite{LB02} for more details.

For a general crystabelian representation $V$, we may choose an integer $h\geq0$ such that $V(-h)$ is positive, and we define $\mathrm{N}(V)=T^{-h}\mathrm{N}(V(-h))$; it is independent of the choice of $h$. We call $\mathrm{N}(V)$ the \emph{Wach module} of $V$.

\begin{proposition}(\cite[Th\'eor\'eme 3.2.1]{BB06})
If $V$ is a positive crystabelian representation, then $\D_{\mathrm{cris}}(V)=(\r_L^+\otimes_{\calE_L^+}\mathrm{N}(V))^{\Gamma_n}$ for $n$ sufficiently large.
\end{proposition}

Thus for a positive crystabelian representation $V$, we have $\r_L^+\otimes_L\D_{\mathrm{cris}}(V)\subseteq\r_L^+\otimes_{\calE_L^+}\mathrm{N}(V)$. Moreover, if the Hodge-Tate weights of $V$ are in the interval $[-h,0]$ for some $h\geq0$, then we have $\r_L^+\otimes_{\calE_L^+}\mathrm{N}(V)\subseteq t^{-h}\r_L^+\otimes_L\D_{\mathrm{cris}}(V)$ \cite[Corollaire 3.2.7]{BB06}.

Using the map $\iota_m$, we get
\[
L_m[[t]]\otimes_L\D_{\mathrm{cris}}(V)\subseteq L_m[[t]]\otimes_{\calE_L^+}^{\iota_m}\mathrm{N}(V)
\subseteq t^{-h}L_m[[t]]\otimes_L\D_{\mathrm{cris}}(V).
\]
We further have the following result \cite[Lemme 3.3.1]{BB06}.
\begin{proposition}
If $m\geq0$, then the image $L_m[[t]]\otimes_{\calE_L^+}^{\iota_m}\mathrm{N}(V)$ in $L_m((t))\otimes_L\D_{\mathrm{cris}}(V)$ is contained in $\mathrm{Fil}^0(t^{-h}L_m[[t]]\otimes_L\D_{\mathrm{cris}}(V))$, and if $m\geq m(V)$, then the map
\[
L_m[[t]]\otimes_{\calE_L^+}^{\iota_m}\mathrm{N}(V)\ra\mathrm{Fil}^0(t^{-h}L_m[[t]]\otimes_L\D_{\mathrm{cris}}(V))
\]
is an isomorphism.
\end{proposition}

\section{$p$-adic local Langlands correspondence for $\G$}

\subsection{Breuil's $p$-adic local Langlands program of $\G$}
In this subsection, we give a sketch of the motivation of Breuil's $p$-adic local Langlands program of $\G$, and we
show that $B(\alpha)/L(\alpha)$ is the admissible unitary representation corresponds to $V_{\alpha,\beta}$ as announced in section 1. The main source of our exposition is Emerton's article \cite{E06}.

Let $l$ be a prime, and let $V$ be a $2$-dimensional continuous representation of $G_{\mathbb{Q}_l}$ over $\overline{\mathbb{Q}}_p$. Applying either the recipe of Deligne (\cite{D73}) if $l\neq p$, or the recipe of Fontaine (\cite{F94}) if $l=p$ and $V$ is potentially semistable, we may attach $V$ a Frobenius semi-simple Weil-Deligne representation $\sigma^{\mathrm{ss}}(V):\mathrm{WD}_{\mathbb{Q}_l}\ra \mathrm{GL}_2(\overline{\mathbb{Q}}_p)$, which corresponds to an admissible smooth representation $\pi_l(V):=\pi_l(\sigma^{\mathrm{ss}}(V))$ of $\mathrm{GL}_2(\mathbb{Q}_l)$ via the classical local Langlands correspondence $\pi_l$.

In case $l\neq p$, Deligne's procedure to construct $\sigma(V)$ from $V$ is convertible. So if $V$ is Frobenius semi-simple (as is conjectured to be the case when $V$ is the restriction to $G_{\mathbb{Q}_l}$ of a global $p$-adic Galois representation attached to a cuspidal newform), then it is determined up to isomorphism by the associated $\mathrm{GL}_2(\mathbb{Q}_l)$-representation $\pi_l(V)$.

On the other hand, if $l=p$ and $V$ is potentially semistable, then the construction of $\sigma(V)$ involves passing to the potentially semistable Dieudonn\'e module $\D_{\mathrm{pst}}(V)$ of $V$, and then forgetting the Hodge filtration. In general, for a given $(\varphi,N,G_{\Q})$-module, one can equip it with an admissible filtration (a filtration so that it becomes an admissible filtered $(\varphi,N,G_{\Q})$-module) in many different ways. Therefore $V$ is usually not uniquely determined by $\pi_p(V)$.

Breuil conjectured that there should be a $p$-adic local Langlands correspondence which attaches $V$ a $p$-adic Banach space representation $\mathrm{B}(V)$. This representation $\mathrm{B}(V)$ should determine $V$ up to isomorphism. (Breuil's original conjecture limited to the case that $V$ is potentially semistable; Colmez constructed this correspondence for all irreducible $V$ later on, as will be explained in subsection 3.2.) For our purpose, we restrict to the case when $V$ has distinct Hodge-Tate weights $k_1<k_2$. Consider the following locally algebraic representation
\[
\tilde{\pi}_p(V):=\pi^{\mathrm{m}}_p(V)\otimes \mathrm{Sym}^{k_2-k_1-1}L^2\otimes\mathrm{det}^{k_1+1}\otimes((x|x|)^{-1}\circ\det),
\]
which encodes the Hodge-Tate weights of $V$, where $\pi^{\mathrm{m}}_p$ is a modified version of the classical local Langlands correspondence for $\mathrm{GL}_2$ introduced by Breuil (for more details about $\pi^{\mathrm{m}}_p$, see \cite[2.1.1]{E06}). Breuil's idea is that the representation $\mathrm{B}(V)$ should be regarded as a completion of $\tilde{\pi}_p(V)$ with respect to certain $\G$-invariant norm, and that this extra data should determine the Hodge filtration uniquely. Note that our definition of $\tilde{\pi}_p(V)$ differs by a twist of $(x|x|)^{-1}\circ\det$ from the definition of $\tilde{\pi}_p(V)$ given in \cite[3.3.1(7)]{E06}. This because Emerton normalizes the $p$-adic local Langlands correspondence for $\G$ by requiring that the central character of $\mathrm{B}(V)$ is equal to $\det V(x|x|)$ (\cite[3.3.1(2)]{E06}). But the normalization chosen by Breuil and Colmez, which is the one we use in this paper, satisfies that the central character of $\mathrm{B}(V)$ is equal to $\det V(x|x|)^{-1}$.

Back to the case $V_{\alpha,\beta}$, if we view $\alpha,\beta$ as characters of $\mathrm{W}_{\Q}^{\mathrm{ab}}$ via the isomorphism $\Q^\times\cong W_{\Q}^{\mathrm{ab}}$ provided by the local Artin map, then we have $\sigma^{\mathrm{ss}}(V_{\alpha,\beta})=\sigma(V_{\alpha,\beta})=L\cdot e_{\alpha}\oplus L\cdot e_{\beta}$ (with trivial monodromy action) by Fontaine's recipe. Recall that if $\alpha\beta^{-1}\neq|x|^{\pm1}$, then we have $\pi^{\mathrm{m}}_p(L\cdot e_{\alpha}\oplus L\cdot e_{\beta})=(\Ind^{\G}_{\mathrm{B}(\Q)}\beta|x|\otimes\alpha)^{\mathrm{sm}}$; while if $\alpha\beta^{-1}=|x|$ (resp. $|x|^{-1}$), then $\pi^{\mathrm{m}}_p(V_{\alpha,\beta})=(\Ind^{\G}_{\mathrm{B}(\Q)}\alpha|x|\otimes\beta)^{\mathrm{sm}}$ (resp. $(\Ind^{\G}_{\mathrm{B}(\Q)}\beta|x|\otimes\alpha)^{\mathrm{sm}}$).  It follows that
\begin{equation*}
\begin{split}
\tilde{\pi}_p(V_{\alpha,\beta})&=\pi^{\mathrm{m}}_p(V_{\alpha,\beta})\otimes \mathrm{Sym}^{k-2}L^2\otimes\mathrm{det}\otimes((x|x|)^{-1}\circ\det)\\
&=\left\{
         \begin{array}{ll}
       \pi(\beta)  & \text{if $\alpha\neq\beta|x|$}\\
          \pi(\alpha)  & \text{if $\alpha=\beta|x|$}
         \end{array}
       \right.\\
&=\pi(\alpha)
\end{split}
\end{equation*}
by intertwining operators. It is not difficult to see that the Hodge filtration of $\D_{\mathrm{cris}}(V_{\alpha,\beta})$ is the only admissible filtration (up to isomorphism) of the $(\varphi,G_{\Q})$-module $\sigma(V_{\alpha,\beta})$ (in fact, for a two dimensional potentially semistable representation $V$ of $G_{\Q}$, the $(\varphi,N,G_{\Q})$-module $\D_{\mathrm{pst}}(V)$ has a unique admissible filtration if and only if $V$ is irreducible and potentially crystalline and $\sigma(V)$ is abelian). Hence we should have $\mathrm{B}(V_{\alpha,\beta})$ to be the universal unitary completion of $\pi(\alpha)$, i.e. $B(\alpha)/L(\alpha)$ by proposition 1.3, according to Breuil's idea. However, a priori it is not clear that whether $B(\alpha)/L(\alpha)$ is nonzero. Inspired by the work of Colmez \cite{C04}, Berger and Breuil showed that $B(\alpha)/L(\alpha)$ is nonzero by means of $\m$-modules as will be explained in subsection 3.4.

\subsection{Colmez's construction of the $p$-adic local Langlands correspondence for $\G$}
We recall Colmez's construction of the $p$-adic local Langlands correspondence for $\G$ and his identification of locally analytic vectors in this subsection. We refer the reader to \cite{C08} for a complete treatment. We start with the notion of products of $(\varphi,\Gamma)$-modules with open subsets of $\Q$. For this see \cite[III.1]{C08c} for more details.

Let $D$ be a finite free \'etale $\m$-module over $\OO_{\calE_L}$, and let $D^\dagger,D_\rig$ be the corresponding \'etale $\m$-modules over $\calE^\dagger_L,\r_L$, respectively. For any $a\in\Z^\times$, let $\sigma_a$ denote the element of $\Gamma$ such that $\chi(\sigma_a)=a$. We equip $D$ with a $P(\Z)=\left(\begin{smallmatrix}\Z^\times&\Z\\0&1\end{smallmatrix}\right)$-action by setting $\left(\begin{smallmatrix}a&b\\0&1\end{smallmatrix}\right)z=(1+T)^b\sigma_a(z)$ for any $z\in D$. For any subset $i+p^n\Z$ of $\Z$, we set $\Res_{i+p^n\Z}(z)=(1+T)^i\varphi^n\psi^n((1+T)^{-i}z)$. This is independent of the choice of the representative $i$. In general, if $U$ is an open compact subgroup of $\Z$, and if $k$ is sufficiently large such that $U$ is a union of some translations of $p^k\Z$, then the $\OO_L$-linear map $\sum_{a\in U\mathrm{mod}p^k\Z}\Res_{a+p^k\Z}$ is independent of the choice of $k$, and we denote it by $\Res_U$. For any $\OO_L$-submodule $M$ of $D$ stable under $P(\Z)$ and $\psi$-actions, we define the $\OO_L$-submodule $M\boxtimes U$ of $D$ as the image $\Res_UM$, which is stable under the $P(\Z)$-action. For example, it is clear that $D\boxtimes\Z=D$ and $D\boxtimes\Z^\times=D^{\psi=0}$.

If $M$ is further stable under $\varphi$, we define $M\boxtimes\Q$ as the set of sequences $(z^{(n)})_{n\in\mathbb{N}}$ of elements of $M$, such that $\psi(z^{(n+1)})=z^{(n)}$ for any $n\in\mathbb{N}$, and we identify $M$ as a submodule of $M\boxtimes\Q$ by sending $z\in M$ to $(\varphi^n(z))_{n\in\mathbb{N}}$. We extend the $P(\Z)$, $\psi$ and $\varphi$-actions to $M\boxtimes\Q$ by the formulas
\[
\left(\begin{smallmatrix}a&b\\0&1\end{smallmatrix}\right)((z^{(n)})_{n\in\mathbb{N}})
=(\left(\begin{smallmatrix}a&b\\0&1\end{smallmatrix}\right)z^{(n)})_{n\in\mathbb{N}}, \psi((z^{(n)})_{n\in\mathbb{N}})=(z^{(n-1)})_{n\in\mathbb{N}},\varphi((z^{(n)})_{n\in\mathbb{N}})
=(z^{(n+1)})_{n\in\mathbb{N}},
\]
where we put $z^{(-1)}=0$.
For $U$ open compact in $\Z$, we define the map $\Res_U:M\boxtimes\Q\ra M\boxtimes\Q$ by the formula
\[
\Res_U((z^{(n)})_{n\in\mathbb{N}})=(\varphi^n(\Res_U(z^{(0)})))_{n\in\mathbb{N}}\in M\boxtimes\Q,
\]
where $\Res_U(z^{(0)})\in M\boxtimes U\subset M$. Thus $\Res_U(M\boxtimes\Q)\subset M\boxtimes U\subset M$, where $M$ is identified as a submodule of $M\boxtimes\Q$ as above. If $U$ is a open compact subset of $\Q$, and if $k\in\mathbb{N}$ such that $p^kU\subset\Z$, then we define $M\boxtimes U\subset M\boxtimes\Q$ and $\Res_U:M\boxtimes\Q\ra M\boxtimes U$ as
\[
M\boxtimes U=\psi^k(M\boxtimes p^kU)\quad\text{and}\quad \Res_U=\psi^k\circ\Res_{p^kU}\circ\varphi^k;
\]
they are independent of the choice of $k$. Moreover, when $U$ is contained in $\Z$, this definition coincides with the definition above, regarding $U$ as a compact open subset of $\Z$. Note that all the constructions above apply to $D[1/p],D^\dagger,D_\rig$.

From now on, we further suppose $\rank_{\OO_{\calE}}D=2$. Then $\wedge^2D$ is of the form $\OO_{\calE}\otimes\delta_D'$ for some continuous character $\delta_D':\Q^\times\ra \OO_L^\times$. Let $\delta_D$ be the character defined by $\delta_D(z)=(z|z|)^{-1}\delta_D'(z)$. If $g=\left(\begin{smallmatrix}a&b \\c&d\end{smallmatrix}\right)\in\G$ and $U$ is open compact in $\Q$ such that $-\frac{d}{c}$ is not in $U$, then we set $g(i)=\frac{ai+b}{ci+d}$ for any $i\in U$. For any $z\in D\boxtimes U$, the operator $H_g:D\boxtimes U\ra D\boxtimes U$ is defined as
\[
H_g(z)=\lim_{n\ra\infty}\sum_{i\in U\mathrm{mod}p^n\Z}\delta_D(ci+d)\left(\begin{smallmatrix}
 g'(i)&g(i) \\
 0&1
\end{smallmatrix}\right)\Res_{p^n\Z}(\left(\begin{smallmatrix}
 1&-i \\
 0&1
\end{smallmatrix}\right)z).
\]
Here $g'(i)=\frac{ad-bc}{(ci+d)^2}$ is the derivative of $g(i)$. Put $w=\w$. Let $w_{D}$ be the restriction of $H_{w}$ on $D\boxtimes\Z^\times$, hence
\[
w_{D}(z)=\lim_{n\ra\infty}\sum_{i\in \Z^\times\mathrm{mod}p^n\Z}\delta_D(i)\left(\begin{smallmatrix}
 -i^{-2}&i^{-1} \\
 0&1
\end{smallmatrix}\right)\Res_{p^n\Z}(\left(\begin{smallmatrix}
 1&-i \\
 0&1
\end{smallmatrix}\right)z).
\]
We define
\[
D\boxtimes\mathbf{P}^1=\{z=(z_1,z_2)\in D\times D, \Res_{\Z^\times}(z_2)=w_D(\Res_{\Z^\times}(z_1))\}.
\]
For any $U$ open compact in $\Q$ and $z=(z_1,z_2)\in D\boxtimes\mathbf{P}^1$, we define $\Res_U(z)\in D\boxtimes U$ by
\[
\Res_U(z)=\Res_{U\cap\Z}(z_1)+H_w(\Res_{wU\cap p\Z}(z_2))=\Res_{U\cap p\Z}(z_1)+H_w(\Res_{wU\cap\Z}(z_2)).
\]
The last equality holds as $\Res_{\Z^\times}(z_2)=w_D(\Res_{\Z^\times}(z_1))$.
\begin{theorem}(\cite[Th\'eor\`{e}me II.1.4]{C08})
There exists a unique $G$-action on $D\boxtimes\mathbf{P}^1$ such that
\[
\Res_U(g\cdot z)=H_g(\Res_{g^{-1}U\cap\Z}(z_1))+H_{gw}(\Res_{(gw)^{-1}U\cap p\Z}(z_2))
\]
for any $g\in\G$ and $U$ open compact in $\Q$.
\end{theorem}
The following proposition describes the $G$-action more precisely.
\begin{proposition}(\cite[Proposition II.1.8]{C08})
The $\G$-action on $D\boxtimes\mathbf{P}^1$ satisfies that if $z=(z_1,z_2)$, then
\begin{enumerate}
\item[(i)]$\w z=(z_2,z_1)$;
\item[(ii)]if $a\in\Q^\times$, then $\left(\begin{smallmatrix}
 a&0 \\
 0&a
\end{smallmatrix}\right)z=(\delta_D(a)z_1,\delta_D(a)z_2)$;
\item[(iii)]if $a\in\Z^\times$, then $\left(\begin{smallmatrix}
 a&0 \\
 0&1
\end{smallmatrix}\right)z=(\left(\begin{smallmatrix}
 a&0 \\
 0&1
\end{smallmatrix}\right)z_1,\delta_D(a)\left(\begin{smallmatrix}a^{-1}&0\\0&1\end{smallmatrix}\right)z_2)$;
\item[(iv)]if $z'=\left(\begin{smallmatrix}p&0\\0&1\end{smallmatrix}\right)z$, then $\Res_{p\Z}z'=\left(\begin{smallmatrix}p&0\\0&1\end{smallmatrix}\right)z_1$ and $\Res_{\Z}\w z'=\delta_D(p)\psi(z_2)$;
\item[(v)]if $b\in p\Z$, and if $z'=\left(\begin{smallmatrix}1&b\\0&1\end{smallmatrix}\right)z$, then $\Res_{\Z}z'=\left(\begin{smallmatrix}1&b\\0&1\end{smallmatrix}\right)z_1$ and $\Res_{p\Z}\w z'=u_b(\Res_{p\Z}(z_2))$, where $u_b=\delta^{-1}(1+b)\left(\begin{smallmatrix}1&-1\\0&1\end{smallmatrix}\right)\circ w_D\circ\left(\begin{smallmatrix}(1+b)^2&b(1+b)\\0&1\end{smallmatrix}\right)\circ w_D\circ \left(\begin{smallmatrix}1&1/(1+b)\\0&1\end{smallmatrix}\right)$ on $D\boxtimes p\Z$.
\end{enumerate}
\end{proposition}

For any $z\in D\boxtimes\mathbf{P}^1$, by \cite[Proposition II.1.14(i)]{C08}, $(\Res_{\Z}\left(\begin{smallmatrix}p^n&0\\0&1\end{smallmatrix}\right))_{n\in\mathbb{N}}$ is an element of $D\boxtimes\Q$; we denote this element by $\Res_{\Q}z$. We define $D^\natural\boxtimes\mathbf{P}^1=\{z\in D\boxtimes\mathbf{P}^1,\Res_{\Q}z\in D^\natural\boxtimes\Q\}$.

Let $\mathrm{Rep}_{\mathrm{tors}}\G$ be the category of smooth $\OO_L[\G]$-modules which are of finite length and admit central characters. Let $\mathrm{Rep}_{\OO_L}\G$ be the category of $\OO_L[\G]$-modules $\Pi$ which are separated and complete for the $p$-adic topology, $p$-torsion free, and satisfies $\Pi/p^n\Pi\in\mathrm{Rep}_{\mathrm{tors}}\G$ for any $n\in\mathbb{N}$.

\begin{theorem}(\cite[Th\'eor\`eme II.3.1]{C08})
Keep notations as above. The following are true.
\begin{enumerate}
\item[(i)]The submodule $D^\natural\boxtimes\mathbf{P}^1$ of $D\boxtimes\mathbf{P}^1$ is stable under $\G$.
\item[(ii)]The representation $\Pi(D)=(D\boxtimes\mathbf{P}^1)/(D^\natural\boxtimes\mathbf{P}^1)$ is an object of $\mathrm{Rep}_{\OO_L}\G$ with central character $\delta_D$, and $D^\natural\boxtimes\mathbf{P}^1$ is naturally isomorphic to $\Pi(D)^\ast\otimes(\delta_D\circ\det)$. Thus we have the following exact sequence
\[
0\longrightarrow\Pi(D)^\ast\otimes(\delta_D\circ\det)\longrightarrow D\boxtimes\mathbf{P}^1\longrightarrow\Pi(D)\longrightarrow0.
\]
\end{enumerate}
\end{theorem}
We denote $\Pi(\check{D})$ by $\check{\Pi}(D)$. Here $\check{D}=\mathrm{Hom}_{\r_L}(D,\calE_L\frac{dT}{1+T})$ is the Tate dual of $D$, where the $\varphi,\Gamma$-actions on $dT/(1+T)$ are defined as $\varphi(dT/(1+T))=dT/(1+T),\gamma(dT/(1+T))=\chi(\gamma)dT/(1+T)$. It is clear that if $D=\D(V)$, then $\check{D}=\D(\check{V})$; here we denote by $\check{V}$ the Tate dual of $V$. Note that $\check{D}\cong D\otimes\delta_D^{-1}$. It follows that $\check{\Pi}(D)\cong\Pi(D)\otimes(\delta_D^{-1}\circ\det)$; so $D^\natural\boxtimes\mathbf{P}^1$ is naturally isomorphic to $(\check{\Pi}(D))^\ast$. The $w_D$-action induces an involution on $D[1/p]\boxtimes\Z^\times$, and the $\G$-action naturally extends to
\[
D[1/p]\boxtimes\mathbf{P}^1=\{(z_1,z_2)\in D[1/p]\times D[1/p], w_D(\Res_{\Z^\times}z_1)=\Res_{\Z^\times}z_2\}.
\]
We set $\Pi(D[1/p])=\Pi(D)[1/p]$ and $\check{\Pi}(D[1/p])=\check{\Pi}(D)[1/p]$; they are admissible unitary representations of $\G$.

If $C$ is a pro-$p$ cyclic group, and if $c$ is a topological generator of $C$, the set of $g(c-1)$ for $g(T)\in\OO_L[[T]]$ is independent of the choice of $c$; the resulting ring is denoted by $\Lambda_L(C)$. For any ring $R$ of $\calE^{(0,r]}_L$, $\calE_L^\dagger$, $\r^+_L$, $\calE_L^{]0,r]}$ and $\r_L$, we define $R(C)$ similarly. Let $\Delta$ be the torsion subgroup of $\Gamma$, then $\Gamma=\Delta\times\Gamma_1$.
We define $\Lambda_L(\Gamma)=\OO_L[\Delta]\otimes\Lambda_L(\Gamma_1)$, and we define $R(\Gamma)=L[\Delta]\otimes R(\Gamma_1)$.
For any $h\geq1$, it is clear that $\Lambda_L(\Gamma)$ (resp. $R(\Gamma)$) is finite free over $\Lambda_L(\Gamma_h)$ (resp. $R(\Gamma_h)$), and $\Lambda_L(\Gamma)$ (resp. $R(\Gamma)$) has a $\Lambda_L(\Gamma_h)$-basis (resp. $R(\Gamma_h)$-basis) consisting of elements in $\Gamma$. For a finite free module $M$ over $\Lambda_L(\Gamma)$ (resp. $R(\Gamma)$) equipped with a continuous semilinear $\Gamma$-action, we define a continuous action of $\Lambda_L(\Gamma_1)$ (resp. $R(\Gamma_1)$) on $M$ by setting
\[
(\sum a_i(\gamma_1-1)^i)(m)=\sum a_i((\gamma_1-1)^i(m))
\]
where $\gamma_1$ is a topological generator of $\Gamma_1$. We further extend this action to a continuous action of $\Lambda_L(\Gamma)$ (resp. $R(\Gamma)$) on $M$ by setting $(f\otimes g)(m)=f(g(m))$.

Suppose $\mathrm{V}(D^\dagger)=V$, and $\dim_LV=d$. Let $D^{\dagger,r}=\D^{\dagger,r}(V)$ and $D^{\dagger,r}_\rig=\D^{\dagger,r}_\rig(V)$. We have the following result \cite[Th\'eor\`eme V.1.12]{C08}.
\begin{theorem}For $r$ sufficiently large, the following are true.
\begin{enumerate}
\item[(i)]If $s\geq r$, then $D^{\dagger,s}\boxtimes\Z^\times$ is a free $\calE_L^{(0,s]}$-module of rank $d$ generated by $D^{\dagger,r}\boxtimes\Z^\times$;
\item[(ii)]If $s\geq r$, then $D^{\dagger,s}_\rig\boxtimes\Z^\times$ is a free $\calE_L^{]0,s]}$-module of rank $d$ generated by $D^{\dagger,r}\boxtimes\Z^\times$.
\end{enumerate}
As a consequence, $D^\dagger\boxtimes\Z^\times$ is a free $\calE^\dagger_L(\Gamma)$-module of rank $d$, and \[
D_\rig\boxtimes\Z^\times=\r_L(\Gamma)\otimes_{\calE^\dagger_L(\Gamma)}D^\dagger\boxtimes\Z^\times
\]
is a free $\r_L(\Gamma)$-module of rank $d$.
\end{theorem}
The following proposition follows from \cite[Lemme V.2.4]{C08}
\begin{proposition}
$D^\dagger\boxtimes\Z^\times$ is stable under the action of $w_D$.
\end{proposition}
For any character $\tau:\Z^\times\ra \OO_L^\times$ and $n\in\mathbb{Z}$, suppose $|\tau(1+p^h\Z)-1|<1$ for some $h\geq1$. Then $\lambda(\gamma-1)\ra\lambda(\tau(\chi(\gamma))\gamma^n-1)$ for any $\lambda(\gamma-1)\in\r_L(\Gamma_{h})$ defines an $L$-linear automorphism on $\r_L(\Gamma_h)$. We can extend this automorphism uniquely to $\r_L(\Gamma)$ by sending $\gamma$ to $\tau(\chi(\gamma))\gamma^n$ for any $\gamma\in\Gamma$. The resulting automorphism on $\r_{L}(\Gamma)$ is independent of the choice of $h$, and we denote it by $T_{\tau,n}$. It is obvious that $T_{\tau_1,n_1}\circ T_{\tau_2,n_2}=T_{\tau_1\tau_2,n_1+n_2}$. We use $T_{\tau}$ to denote $T_{\tau,0}$ for simplicity. Both $\r^+_L(\Gamma)$ and $\calE^\dagger_L(\Gamma)$ are stable under the action of $T_{\tau,n}$.

Applying the proposition above, we extend the action of $w_D$ to $D_\rig\boxtimes\Z^\times=\r_L(\Gamma)\otimes_{\calE^\dagger_L(\Gamma)}D^\dagger\boxtimes\Z^\times$ by the formula $w_D(\lambda\otimes z)=T_{\delta_D,-1}(\lambda)\otimes w_D(z)$ for $\lambda\in\r_L(\Gamma)$ and $z\in D^\dagger\boxtimes\Z^\times$. Then we define
\[
D_\rig\boxtimes\mathbf{P}^1=\{(z_1,z_2)\in D_\rig\times D_\rig, \Res_{\Z^\times}z_2=w_D(\Res_{\Z^\times} z_1)\}.
\]

\begin{proposition}(\cite[Propositions V.2.8, V.2.9]{C08})
\begin{enumerate}
\item[(i)]$D^\dagger\boxtimes\mathbf{P}^1=\{(z_1,z_2)\in D\boxtimes\mathbf{P}^1,z_1,z_2\in D^\dagger\}$
is stable under the action of $\G$;
\item[(ii)]The $\G$-action on $D^\dagger\boxtimes\mathbf{P}^1$ extends to a continuous $\G$-action on $D_\rig\boxtimes\mathbf{P}^1$ satisfying the formulas listed in Proposition 3.2.
\end{enumerate}
\end{proposition}
By \cite[Th\'eor\`eme I.5.2]{C08}, we know that $(1-\varphi)D^{\psi=1}$ is a free $\Lambda_L(\Gamma)$-module of rank $d$. The following proposition will be used in subsection 4.1.
\begin{proposition}(\cite[Corollaire V.1.6(iii)]{C08})
The inclusion $(1-\varphi)D^{\psi=1}\subset D_\rig\boxtimes\Z^\times$ induces an isomorphism from $\r_L(\Gamma)\otimes_{\Lambda_L(\Gamma)}(1-\varphi)D^{\psi=1}$ to $D_\rig\boxtimes\Z^\times$.
\end{proposition}

For $\omega=gdT$ a differential $1$-form with $g=\sum_{k\in\mathbb{Z}}a_kT^k\in\calE_L$, we define the residue $\res_0(w)=a_{-1}$.
We define the  pairing $\{$ , $\}$:$\check{D}\times D\ra L$ by the formula
\[
\{x,y\}=\res_0((\sigma_{-1}\cdot x)(y)).
\]
We further extend $\{$ , $\}$ to a pairing $\{$ , $\}_{\mathbf{P}^1}:\check{D}\boxtimes\mathbf{P}^1\times D\boxtimes\mathbf{P}^1\ra L$ by the formula
\[
\{(z_1,z_2),(z_1',z_2')\}_{\mathbf{P}^1}=\{z_1,z_1'\}+\{\Res_{p\Z}z_2,\Res_{p\Z}z_2'\}.
\]
\begin{theorem}(\cite[Th\'eor\`eme II.1.13]{C08})
The pairing $\{$ , $\}$ is perfect and $\G$-equivariant.
\end{theorem}
\begin{theorem}(\cite[Th\'eor\`eme II.2.11]{C08})
$D^\natural\boxtimes\mathbf{P}^1$ and $\check{D}^\natural\boxtimes\mathbf{P}^1$ are orthogonal complements of each other.
\end{theorem}
We define the pairing $\{$ , $\}_{\mathbf{P}^1}:\check{D}_\rig\boxtimes\mathbf{P}^1\times D_\rig\boxtimes\mathbf{P}^1\ra L$ similarly; it is also perfect and $\G$-equivariant. Let $D^\natural_{\rig}\boxtimes\mathbf{P}^1$ denote the orthogonal complement of $\check{D}^\natural\boxtimes\mathbf{P}^1$ in $D_\rig\boxtimes\mathbf{P}^1$ with respect to $\{$ , $\}_{\mathbf{P}^1}$.
\begin{theorem}(\cite[Th\'eor\`eme V.2.12]{C08})
\begin{enumerate}
\item[(i)]$\Pi(D)_{\mathrm{an}}=(D^\dagger[1/p]\boxtimes\mathbf{P}^1)/(D^\natural[1/p]\boxtimes\mathbf{P}^1)$ and $D^\natural_{\rig}\boxtimes\mathbf{P}^1=(\check{\Pi}(D)_{\mathrm{an}})^\ast$.
\item[(ii)]The natural map $(D^\dagger[1/p]\boxtimes\mathbf{P}^1)/(D^\natural[1/p]\boxtimes\mathbf{P}^1)\ra
(D_\rig\boxtimes\mathbf{P}^1)/(D^\natural_\rig\boxtimes\mathbf{P}^1)$ is an isomorphism.
\end{enumerate}
\end{theorem}

For $V$ a $2$-dimensional $L$-linear representation of $G_{\Q}$, we set $\Pi(V)=\Pi(\D(V))$ and $\check{\Pi}(V)=\Pi(\D(\check{V}))$.

\subsection{Amice transformation}
For any $h\in\mathbb{N}$, let $\mathrm{LA}_h$ denote the space of functions $f:\Z\ra L$ such that $f$ is analytic on $a+p^h\Z$ for any $a\in\Z$. If $f\in\mathrm{LA}_h$, then for any $z_0\in\Z$, we expand $f$ on $z_0+p^h\Z$ in the form
\[
f(z)|_{z_0+p^h\Z}=\sum_{i=0}^\infty a_{h,i}(z_0)(\frac{z-z_0}{p^h})^i,
\]
where $a_{h,i}(z_0)$ is a sequence of elements in $L$ such that $|a_{h,i}|\ra 0$ as $i\ra\infty$. We set $\|f\|_{h,z_0}=\max_i\{|a_{h,i}|\}$ and $\|f\|_{\mathrm{LA}_h}=\sup_{z_0\in\Z}\|f\|_{z_0,h}$. Let $\mathrm{LA}=\cup_h \mathrm{LA}_h$ denote the space of $L$-valued locally analytic functions on $\Z$. A \emph{continuous distribution} on $\Z$ is an $L$-linear homomorphism from $\mathrm{LA}$ to $L$ such that the restriction to each $\mathrm{LA}_h$ is continuous. Let $\mathcal{D}_{\mathrm{cont}}(\Z,L)$ denote the set of continuous distributions on $\Z$. We set, for any $h\in\mathbb{N}$, a norm $\|\cdot\|_{\mathrm{LA}_h}$ on $\mathcal{D}_{\mathrm{cont}}(\Z,L)$ by the formula
\[
\|\mu\|_{\mathrm{LA}_h}=\sup_{f\in\mathrm{LA}_h-{0}}\frac{|\int_{\Z}fd\mu(z)|}{\|f\|_{\mathrm{LA}_h}}
\]
We equip $\mathcal{D}_{\mathrm{cont}}(\Z,L)$ with the Fr\'echet topology defined by the norms $\|\cdot\|_{\mathrm{LA}_h}$ for $h\in\mathbb{N}$.

Let $\mu\in\mathcal{D}_{\mathrm{cont}}(\Z,L)$. For any $\gamma\in\Gamma$, we define $\gamma(\mu)\in\mathcal{D}_{\mathrm{cont}}(\Z,L)$ by the formula
\[
\int_{\Z}f(z)d\gamma(\mu)(z)=\int_{\Z}f(\chi(\gamma)z)d\mu(z).
\]
We define $\varphi(\mu),\psi(\mu)\in\mathcal{D}_{\mathrm{cont}}(\Z,L)$ by the formulas
\[
\int_{\Z}f(z)d\varphi(\mu)(z)=\int_{\Z}f(pz)d\mu(z),\int_{\Z}f(z)d\psi(\mu)(z)=\int_{p\Z}f(\frac{z}{p})d\mu(z),
\]
respectively. It is clear that $\psi(\varphi(\mu))=\mu$ and $\varphi(\psi(\mu))$ is the restriction of $\mu$ on $p\Z$.

For any $\mu\in\mathcal{D}_{\mathrm{cont}}(\Z,L)$, we associate it with the Amice transformation $\A(\mu)$, which is an element of $\r^+_L$ defined as
\[
\A(\mu)=\sum_{n=0}^\infty T^n\int_{\Z}{{z}\choose{n}}d\mu(z)=\int_{\Z}(1+T)^zd\mu(z).
\]
If $h\in\mathbb{N}$, put $\rho_{h}=p^{-\frac{1}{(p-1)p^h}}$. Note that $\rho_{h}=|\eta-1|$ for any $\eta\in\mu_{p^{h+1}}$.
\begin{proposition}(Amice transformation)
The map $\mu\ra\A(\mu)$ is a topologically isomorphism from $\mathcal{D}_{\mathrm{cont}}(\Z,L)$ to $\r^+_L$ respecting the $\varphi,\Gamma$ and $\psi$-actions. Moreover, we have
\[
\|\A(\mu)\|_{\rho_h}\leq \|\mu\|_{\mathrm{LA}_h}\leq p\|\A(\mu)\|_{\rho_{h+1}}.
\]
\end{proposition}
\begin{proof}
It is straightforward to verify that the Amice transformation commutes with $\varphi,\Gamma$ and $\psi$-actions. We leave it as an exercise for the reader. The rest is exactly \cite[Th\'eor\`eme 2.3]{C05}.
\end{proof}
Thus for any $\mu\in\mathcal{D}_{\mathrm{cont}}(\Z,L)$, we have
\[
\A(\mu)\in(\r^+_L)^{\psi=0}\Longleftrightarrow0=\varphi(\psi(\A(\mu)))=\A(\varphi(\psi(\mu)))\Longleftrightarrow
\varphi(\psi(\mu))=0
\Longleftrightarrow\mathrm{Supp}(\mu)\subseteq\Z^\times.
\]
We will need this equivalence later.

\subsection{$\mathrm{B}(V_{\alpha,\beta})\cong\Pi(V_{\alpha,\beta})$}
In this subsection, we will explain the compatibility of Colmez and Berger-Breuil's constructions in the case when $V\in\mathscr{S}_*^{\mathrm{cris}}$ is not exceptional (for $V_{\alpha,\beta}$, this is equivalent to $\alpha\neq\beta$). Since every element of $\mathscr{S}_*^{\mathrm{cris}}$ is a twist of $V_{\alpha,\beta}$ for some $(\alpha,\beta)$, it reduces to show that $\mathrm{B}(V_{\alpha,\beta})$ is naturally isomorphic to $\Pi(V_{\alpha,\beta})$ for any $(\alpha,\beta)$ such that $\alpha\neq\beta$. This is the main result of \cite{BB06}. First note that the central character of $B(\alpha)/L(\alpha)$ is $\delta(z)=(\alpha\beta)(z)|z|^{-1}z^{k-2}$, which coincides with the central character $\delta_{D}$ (here $D=\D(V_{\alpha,\beta})$) of $\Pi(V_{\alpha,\beta})$. From now on, we suppose $\alpha\neq\beta$.
\begin{definition}
For any crystabelian representation $V$, we define $\mathrm{M}(V)$ as the set of elements $g\in \r^+_L[1/t]\otimes_L \D_{\mathrm{cris}}(V)$ such that $\iota_m(g)\in \mathrm{Fil}^0(L_m((t))\otimes_L\D_{\rm cris}(V))$ for every $m\geq m(V)$.
\end{definition}

\begin{proposition}(\cite[Proposition 3.3.3]{BB06})
If $V$ is a crystabelian representation with Hodge-Tate weights in $[-h,0]$ for some $h\geq0$, then the $\r_L^+$-module $\mathrm{M}(V)$ is free of rank $\dim_LV$, and it satisfies
\[
T^{-h}\r^+_L\otimes_{\calE_L^+}\mathrm{N}(V)\subseteq\mathrm{M}(V)\subseteq\varphi^{m(V)-1}(T)^{-h}
\r^+_L\otimes_{\calE_L^+}\mathrm{N}(V).
\]
\end{proposition}

\begin{corollary}
The $\r^+_L$-module $\mathrm{M}(V_{\alpha,\beta})$ is contained in $\D^\dagger_{\rig}(V_{\alpha,\beta})$.
\end{corollary}
\begin{proof}
Applying the above proposition to the positive crystabelian representation $V_{\alpha,\beta}(1-k)$, we get $\mathrm{M}(V_{\alpha,\beta}(1-k))\subseteq\varphi^{m(V_{\alpha,\beta})-1}(T)^{1-k}\r^+_L\otimes_{\calE^+}\mathrm{N}
(V_{\alpha,\beta}(1-k))\subseteq
\D^\dagger_{\rig}(V_{\alpha,\beta}(1-k))$.
Since $\r^+_L\otimes_L \D_{\mathrm{cris}}(V_{\alpha,\beta})=t^{1-k}\r^+_L\otimes_L \D_{\mathrm{cris}}(V_{\alpha,\beta}(1-k))$
and $\mathrm{Fil}^0(L_m[[t]]\otimes_L\D_{\rm cris}(V_{\alpha,\beta}))=\mathrm{Fil}^0(L_m[[t]]\otimes_L\D_{\rm cris}(V_{\alpha,\beta}(1-k)))$, we conclude $\mathrm{M}(V_{\alpha,\beta})\subseteq\D^\dagger_{\rig}(V_{\alpha,\beta})$.
\end{proof}

\begin{lemma} (\cite[Lemme 5.1.2]{BB06})
Let $m\geq m(V_{\alpha,\beta})$, and $c_\alpha,c_\beta\in\r^+_L$. Let $\mu_\alpha=\A^{-1}(c_\alpha),\mu_\beta=\A^{-1}(c_\beta)$ denote the corresponding locally analytic distributions over $\Z$. Then the condition
\[
\iota_m(c_\alpha e_{\alpha}+c_\beta e_{\beta})\in \mathrm{Fil}^0(L_m[[t]]\otimes_L\D_{\mathrm{cris}}(V_{\alpha,\beta}))
\]
is equivalent to
\[
G(\beta^{-1}\alpha,\mathbb{\eta}_{p^m}^{p^{m-m(V)}})\alpha_p^{m}\int_{\Z}z^j
\eta^z_{p^m}d\mu_\alpha(z)=\beta_p^{m}\int_{\Z}z^j\eta^z_{p^m}d\mu_\beta(z)
\]
for every $j\in\{0,\dots,k-2\}$ and every primitive $p^m$-th roots of unity $\eta_{p^m}$ in $\overline{\mathbb{Q}}_p$.
\end{lemma}

\begin{corollary}
Let $\mu_\alpha\in A(\alpha)^\ast$ and $\mu_\beta\in A(\beta)^\ast$. We regard $\mu_\alpha|_{\Z}$, $\mu_\beta|_{\Z}$ as elements of $\mathcal{D}(\Z,L)$, and let $c_\alpha=\A(\mu_\alpha|_{\Z}), c_\beta=\A(\mu_\beta|_{\Z})$. If $\mu_\alpha$ and $\mu_\beta$ are related by the condition
\begin{equation}
\int_{\Q}fd\mu_\beta(z)=\frac{1}{C(\alpha_p,\beta_p)}\int_{\Q}I(f)d\mu_\alpha(z)
\end{equation}
for any $f\in\pi(\beta)$, then we have $c_\alpha e_\alpha+c_\beta e_\beta\in \mathrm{M}(V_{\alpha,\beta})$. Here $C(\alpha_p,\beta_p)$ is the constant we defined in 1.3.
\end{corollary}
\begin{proof}
For any $0\leq j\leq k-2$, $y\in\Q^\times$ such that $\val(y)\leq-m(V)$, by $(3.6)$ and Lemma 1.6, we have
\begin{equation}
\begin{split}
\int_{\Z}z^je^{2\pi izy}d\mu_\beta(z)&=\int_{\Q}\frac{1}{C(\alpha_p,\beta_p)}I(1_{p^n\Z}\cdot z^je^{2\pi izy})d\mu_\alpha(z)\\
&=G(\beta^{-1}\alpha,e^{\frac{2\pi iy}{p^{\val(y)+m(V)}}})(\frac{\beta_p}{\alpha_p})^{\val(y)}\int_{\Z}z^je^{2\pi izy}d\mu_\alpha(z).
\end{split}
\end{equation}
Now for any $m\geq m(V_{\alpha,\beta})$ and a primitive $p^m$-th root of unity $\eta_{p^m}$, we choose $y_0$ such that $e^{2\pi iy_0}=\eta_{p^m}$; so $\val(y_0)=-m\leq-m(V_{\alpha,\beta})$. Setting $y=y_0$ in $(3.7)$, we obtain
\begin{equation}
\begin{split}
\beta_p^{m}\int_{\Z}z^j\eta^z_{p^m}d\mu_\beta(z)=G(\beta^{-1}\alpha,\eta_{p^m}^{p^{m-m(V)}})\alpha_p^{m}\int_{\Z}z^j\eta_{p^m}^{z}
d\mu_\alpha(z).
\end{split}
\end{equation}
We conclude $c_\alpha e_\alpha+c_\beta e_\beta\in \mathrm{M}(V_{\alpha,\beta})$ by Lemma 3.15.
\end{proof}

For any $g\in(\r^+_L)^{\psi=0}$, we set $\left(\begin{smallmatrix}
 0&1 \\
 1&0
\end{smallmatrix}\right)_\alpha g=\A(\left(\begin{smallmatrix}
 0&1 \\
 1&0
\end{smallmatrix}\right)(\A^{-1}(g)))$ (resp. $\left(\begin{smallmatrix}
 0&1 \\
 1&0
\end{smallmatrix}\right)_\beta g=\A(\left(\begin{smallmatrix}
 0&1 \\
 1&0
\end{smallmatrix}\right)(\A^{-1}(g)))$) $\in(\r^+_L)^{\psi=0}$ where we regard $\A^{-1}(g)$ as an element of $A(\alpha)^\ast$ (resp. $A(\beta)^\ast$) supported in $\Z^\times$.

Suppose $z=c_\alpha e_\alpha+c_\beta e_\beta\in\D^\dagger_\rig(V_{\alpha,\beta})\boxtimes\Z^\times\cap(\r^+_Le_\alpha\oplus\r^+_Le_\beta)$. We would have $0=\psi(z)=\alpha_p\psi(c_\alpha)e_\alpha+\beta_p\psi(c_\beta)e_\beta$, yielding $c_\alpha,c_\beta\in(\r^+_L)^{\psi=0}$. We define
\[\left(\begin{smallmatrix}
 0&1 \\
 1&0
\end{smallmatrix}\right)z=\left(\begin{smallmatrix}
 0&1 \\
 1&0
\end{smallmatrix}\right)_\alpha(c_\alpha)e_\alpha+\left(\begin{smallmatrix}
 0&1 \\
 1&0
\end{smallmatrix}\right)_\beta(c_\beta)e_\beta.
\]

We now construct a map $\mathcal{F}$ from $(B(\alpha)/L(\alpha))^\ast$ to
\[
\Pi(V_{\alpha,\beta})^\ast\cong
(\D^\natural(V_{\alpha,\beta})\boxtimes\mathbf{P}^1)\otimes(\delta^{-1}\circ\det).
\]
Let $i_{\alpha}$ denote the natural morphism $A(\alpha)\ra B(\alpha)/L(\alpha)$; the dual map is denoted by $i_{\alpha}^\ast$. We set $i_\beta$ and $i_\beta^\ast$ similarly. For any $\mu_\alpha\in(B(\alpha)/L(\alpha))^\ast$, we associate $\mu_\alpha$ with $\mu_\beta=\frac{1}{C(\alpha_p,\beta_p)}\mu_\alpha\circ\widehat{I}\in(B(\beta)/L(\beta))^\ast$.  We regard $\mu_\alpha$ and $\mu_\beta$ as elements of $A(\alpha)^\ast$ and $A(\beta)^\ast$ via $i_\alpha^\ast$ and $i_\beta^\ast$ respectively. Suppose $c_\alpha=\A(\mu_\alpha|_{\Z}), c_\beta=\A(\mu_\beta|_{\Z})$ and $c_\alpha'=\A((\left(\begin{smallmatrix}
 0&1 \\
 1&0
\end{smallmatrix}\right)\mu_\alpha)|_{\Z}), c_\beta'=\A((\left(\begin{smallmatrix}
 0&1 \\
 1&0
\end{smallmatrix}\right)\mu_\beta)|_{\Z})$. Let $z_\alpha=c_\alpha e_\alpha+c_\beta e_\beta$ and  $z'_\alpha=c'_\alpha e_\alpha+c'_\beta e_\beta$. By Corollary 3.14 we first have $z_\alpha\in\mathrm{M}(V_{\alpha,\beta})$. From \cite[Lemme 5.2.6]{BB06}, the fact that $\mu_\alpha$ and $\mu_\beta$ are of orders $\val(\alpha_p)$ and $\val(\beta_p)$ respectively further ensures that $z_\alpha\in \D^\sharp(V_{\alpha,\beta})$. From \cite[Corollaire II.5.21]{C07}, we get $\D^\sharp(V_{\alpha,\beta})=\D^\natural(V_{\alpha,\beta})$ because $V_{\alpha,\beta}$ is irreducible; hence $z_\alpha\in \D^\natural(V_{\alpha,\beta})$. Similarly we have $z'_\alpha\in\D^\natural(V_{\alpha,\beta})$ because $\widehat{I}$ is $\G$-equivariant.

\begin{lemma}(\cite[Lemme II.3.13]{C08})
For any $z\in\D^\natural({V_{\alpha,\beta}})\boxtimes\Z^\times$, we have $w_D(z)=\left(\begin{smallmatrix}
 0&1 \\
 1&0
\end{smallmatrix}\right)(z)$.
\end{lemma}
Note that $\D^\natural({V_{\alpha,\beta}})\subseteq\mathrm{M}(V_{\alpha,\beta})\subseteq\r^+_Le_\alpha\oplus\r^+_Le_\beta$ following \cite[Corollaire 3.3.10]{BB06}. So $\left(\begin{smallmatrix}
 0&1 \\
 1&0
\end{smallmatrix}\right)(z)$ is defined for any $z\in\D^\natural({V_{\alpha,\beta}})\boxtimes\Z^\times$. By the definition of $z_\alpha$ and $z'_\alpha$ we see that $\Res_{\Z^\times}z'_{\alpha}=\left(\begin{smallmatrix}
 0&1 \\
 1&0
\end{smallmatrix}\right)\Res_{\Z^\times}z_\alpha=w_D(\Res_{\Z^\times}z_\alpha)$. Hence $(z_\alpha,z'_\alpha)$ is an element of $\D^\natural({V_{\alpha,\beta}})\boxtimes\mathbf{P}^1$. We pick a basis $e$ of the one dimensional representation $\delta^{-1}\circ\det$. We define $\mathcal{F}$ by setting $\mathcal{F}(\mu_{\alpha})=(z_\alpha,z'_\alpha)\otimes e$. The following result which is the combination of \cite[Proposition 3.4.6]{BB06} and \cite[Proposition II.3.8]{C08}.
\begin{theorem}
The dual of $\mathcal{F}$ is a topological isomorphism from $\Pi(V_{\alpha,\beta})$ to $(B(\alpha)/L(\alpha))$ as $L$-Banach space representations of $\G$. Furthermore, the $\B$-action on $B(\alpha)/L(\alpha)$ is topologically irreducible.
\end{theorem}
\begin{corollary}
$B(\alpha)/L(\alpha)$ is nonzero.
\end{corollary}
\begin{proof}
For any rank 2 \'etale $\m$-module $D$ over $\calE_L$, $D^\natural\boxtimes\mathbf{P}^1$ is nonzero because $D^\natural\boxtimes\mathbf{P}^1$ contains $D^\natural\boxtimes\Z=D^\natural$. Therefore $(B(\alpha)/L(\alpha))^\ast$ is nonzero, yielding that $B(\alpha)/L(\alpha)$ is nonzero.
\end{proof}

\section{Determination of locally analytic vectors}
We keep assuming $\alpha\neq\beta$ in this section. Let $i_\alpha,i_\beta$ denote the natural maps $A(\alpha)\ra B(\alpha)/L(\alpha),A(\beta)\ra B(\beta)/L(\beta)$ respectively. Since $A(\alpha)$, $A(\beta)$ are locally analytic representations of $\G$, both maps $i_\alpha$ and $\widehat{I}\circ i_\beta$ factor through $(B(\alpha)/L(\alpha))_\mathrm{an}=\mathrm{B}(V_{\alpha,\beta})_{\mathrm{an}}$. It is clear that the map $ i_\alpha\oplus\widehat{I}\circ i_\beta$$:$$A(\alpha)\oplus A(\beta)\ra\mathrm{B}(V_{\alpha,\beta})_{\mathrm{an}}$ reduces to a map $i_\alpha\oplus\widehat{I}\circ i_\beta: A(\alpha)\oplus_{\pi(\beta)}A(\beta)\ra\mathrm{B}(V_{\alpha,\beta})_{\mathrm{an}}$, where we map $\pi(\beta)$ to $A(\alpha)$ via the intertwining operator $I$. Note that if $\alpha=\beta|x|$, since $\ker I=(\beta\circ\det)\otimes_L\mathrm{Sym}^{k-2}L^2$ and $\pi(\beta)/(\beta\circ\det)\otimes_L\mathrm{Sym}^{k-2}L^2=((\beta\circ\det)\otimes_L\mathrm{Sym}^{k-2}L^2)\otimes_L{\mathrm{St}}$ by $(1.5)$, we further have $A(\alpha)\oplus_{\pi(\beta)}A(\beta)= A(\alpha)\oplus_{((\beta\circ\det)\otimes_L\mathrm{Sym}^{k-2}L^2)\otimes_L{\mathrm{St}}}(A(\beta)/((\beta\circ\det)\otimes_L\mathrm{Sym}^{k-2}L^2))$.
The main result of this paper is the following theorem.
\begin{theorem}
If $\alpha\neq\beta$, then the map $i_{\alpha,\beta}=i_\alpha\oplus\widehat{I}\circ i_\beta: A(\alpha)\oplus_{\pi(\beta)}A(\beta)\ra\mathrm{B}(V_{\alpha,\beta})_{\mathrm{an}}$ is a topological isomorphism.
\end{theorem}
This section is devoted to the proof of Theorem 4.1.
\subsection{Extension of $\mathcal{F}$}
Let $i$ denote the inclusion $\mathrm{B}(V_{\alpha,\beta})_{\mathrm{an}}\ra\mathrm{B}(\check{V}_{\alpha,\beta})$. In this subsection we will construct a continuous $\G$-equivariant morphism $\mathcal{F}_\mathrm{an}:(A(\alpha)\oplus_{\pi(\beta)}A(\beta))^\ast\ra\D^\dagger_\rig({V_{\alpha,\beta}})\boxtimes\mathbf{P}^1$ satisfying the following commutative diagram
\begin{equation}
\xymatrix{
(B(\alpha)/L(\alpha))^\ast\ar^{(i\circ i_{\alpha,\beta})^\ast}[d] \ar^{\mathcal{F}}[r]& \D^\natural(\check{V}_{\alpha,\beta})\boxtimes\mathbf{P}^1\ar[d] \\
 (A(\alpha)\oplus_{\pi(\beta)}A(\beta))^\ast \ar^{\mathcal{F}_\mathrm{an}}[r]&\D^\dagger_\rig(\check{V}_{\alpha,\beta})\boxtimes\mathbf{P}^1.}
\end{equation}

Let $\delta_\alpha,\delta_\beta:\Q^\times\ra L^\times$ be the characters defined as $\delta_\alpha(z)=(\beta\alpha^{-1})(z)|z|^{-1}z^{k-2},\delta_\beta(z)=(\alpha\beta^{-1})(z)|z|^{-1}z^{k-2}$.

\begin{lemma}
For $h\in\mathbb{N}$ and $h\geq n(\beta\alpha^{-1})$, we have $\|\w_\alpha(g)\|_{\rho_h}\leq p\|g\|_{\rho_{h+1}}$ and $\|\w_\beta(g)\|_{\rho_h}\leq p\|g\|_{\rho_{h+1}}$ for any $g\in(\r^+_L)^{\psi=0}$. As a consequence, both $\w_{\alpha}$ and $\w_{\beta}$ are continuous with respect to the Fr\'echet topology of $\r_L^+$.
\end{lemma}
\begin{proof}
Let $\mu_\alpha=\A^{-1}(g)$. We regard $\mu_\alpha$ as an element of $A(\alpha)^\ast$. For any $f\in A(\alpha)$, we have
\begin{equation}
\begin{split}
\int_{\Q}f(z)d(\left(\begin{smallmatrix}
 0&1 \\
 1&0
\end{smallmatrix}\right)\mu_\alpha)(z)&=\int_{\Q}(\w(1_{\Z^\times}\cdot f))(z)d\mu_\alpha(z)\\
&=\int_{\Z^\times}\beta(-1)(-1)^k\delta_\alpha(z)f(1/z)d\mu_\alpha(z).
\end{split}
\end{equation}
Thus for any $a\in\Z^\times$, $h\geq n(\beta\alpha^{-1})$ and $m\geq0$, it follows that
\begin{equation}
\begin{split}
\int_{{a+p^h\Z}}(\frac{z-a}{p^h})^md(\left(\begin{smallmatrix}
 0&1 \\
 1&0
\end{smallmatrix}\right)\mu_\alpha)(z)&=\int_{{a^{-1}+p^h\Z}}\beta(-1)(-1)^k\delta_\alpha(z)(\frac{1/z-a}{p^h})^md\mu_\alpha(z)\\
&=\beta(-1)(-1)^k\beta\alpha^{-1}(a^{-1})\int_{a^{-1}+p^h\Z}z^{k-2}(\frac{1/z-a}{p^h})^md\mu_\alpha(z).
\end{split}
\end{equation}
From
\[
1_{a^{-1}+p^h\Z}\cdot(\frac{1/z-a}{p^h})^m=
1_{a^{-1}+p^h\Z}\cdot(\frac{\frac{a}{1+a(z-a^{-1})}-a}{p^h})^m=1_{a^{-1}+p^h\Z}
\cdot(\sum_{i=1}^{\infty}
p^{h(i-1)}a^{i+1}(\frac{z-a^{-1}}{p^h})^i)^m,
\]
we get $\|1_{a^{-1}+p^h\Z}\cdot(\frac{1/z-a}{p^h})^m\|_{\mathrm{LA}_h}\leq1$, yielding
\[
\|1_{a^{-1}+p^h\Z}\cdot z^{k-2}(\frac{1/z-a}{p^h})^m\|_{\mathrm{LA}_h}\leq
\|1_{a^{-1}+p^h\Z}\cdot(\frac{1/z-a}{p^h})^m\|_{\mathrm{LA}_h}\cdot
\|z^{k-2}\|_{\mathrm{LA}_h}
\leq1.
\]
This implies $|\int_{{a+p^h\Z}}(\frac{z-a}{p^h})^md(\left(\begin{smallmatrix}
 0&1 \\
 1&0
\end{smallmatrix}\right)\mu_\alpha)(z)|\leq\|\mu_\alpha\|_{\mathrm{LA}_h}$. Hence $\|\w\mu_\alpha\|_{\mathrm{LA}_h}\leq\|\mu_\alpha\|_{\mathrm{LA}_h}$ (in fact we have $\|\w\mu_\alpha\|_{\mathrm{LA}_h}=\|\mu_\alpha\|_{\mathrm{LA}_h}$ since $\w$ is an involution). So by Proposition 3.11, we get
\[
\|\A(\w\mu_\alpha)\|_{\rho_h}\leq \|\w\mu_\alpha\|_{\mathrm{LA}_h}\leq\|\mu_\alpha\|_{\mathrm{LA}_h}\leq p\|\mathcal{A}(\mu_\alpha)\|_{\rho_{h+1}},
\]
i.e. $\|\w_\alpha(g)\|_{\rho_h}\leq p\|g\|_{\rho_{h+1}}$. We get $\|\w_\beta(g)\|_{\rho_h}\leq p\|g\|_{\rho_{h+1}}$ similarly.
\end{proof}

\begin{lemma}
For any $\lambda\in\r^{+}_L(\Gamma)$ and $g\in(\r^+_L)^{\psi=0}$, we have $\lambda(\left(\begin{smallmatrix}
 0&1 \\
 1&0
\end{smallmatrix}\right)_\alpha g)=\left(\begin{smallmatrix}
 0&1 \\
 1&0
\end{smallmatrix}\right)_\alpha(T_{\delta_\alpha,-1}(\lambda)(g))$ and $\lambda(\w_\beta g)=\w_\beta(T_{\delta_\beta,-1}(\lambda)(g))$.
\end{lemma}
\begin{proof}
Let $\mu_\alpha=\A^{-1}(g)$, regarding as an element of $A(\alpha)^\ast$. For any $\gamma\in\Gamma$, we have
\begin{equation}
\begin{split}
\gamma(\w_\alpha g)&=\gamma\big(\int_{\Z}(1+T)^zd(\w\mu_\alpha)(z)\big)\\
&=\gamma\big(\int_{\Z^\times}\beta(-1)(-1)^k\delta(z)(1+T)^{1/z}d\mu(z)\big) \quad (\text{by $(4.2)$})\\
&=\int_{\Z^\times}\beta(-1)(-1)^k\delta(z)(1+T)^{\chi(\gamma)/z}d\mu(z)\\
&=\int_{\Z^\times}\delta_\alpha(\chi(\gamma))\beta(-1)(-1)^k\delta(z)(1+T)^{1/z}d(\gamma^{-1}\mu)(z)\\
&=\int_{\Z}\delta_\alpha(\chi(\gamma))(1+T)^zd(\w_\alpha(\gamma^{-1}\mu_\alpha))\\
&=\w_\alpha(T_{\delta_\alpha,-1}(\gamma)(g))\end{split}
\end{equation}
So the lemma holds for $\lambda=\gamma$. Let $h=n(\beta\alpha^{-1})$. It reduces to prove the lemma for any $\lambda\in\r^+_L(\Gamma_h)$. Let $\gamma$ be a topological generator of $\Gamma_{h}$. In general, for any $\lambda=\sum_{i=0}^\infty a_i(\gamma-1)^i\in\r^+_L(\Gamma_h)$, we first have
\[
\lambda(\w_\alpha g)=\lim_{j\ra\infty}\sum_{i=0}^ja_i(\gamma-1)^i(\w_\alpha g)=\lim_{j\ra\infty}\w_\alpha(T_{\delta_\alpha,-1}(\sum_{i=0}^ja_i(\gamma-1)^i)(g)).
\]
Since $\lim_{j\ra\infty}\sum_{i=0}^jT_{\delta_\alpha,-1}(a_i(\gamma-1)^i)(g)=T_{\delta_\alpha,-1}(\lambda)(g)$, applying Lemma 4.2, we get
\[
\lim_{j\ra\infty}\w_\alpha(T_{\delta_\alpha,-1}(\sum_{i=0}^ja_i(\gamma-1)^i)(g)=\w_\alpha(\lim_{j\ra\infty}\sum_{i=0}^jT_{\delta_\alpha,-1}
(a_i(\gamma-1)^i)(g)=\w_\alpha(T_{\delta_\alpha,-1}(\lambda)(g)).
\]
So $\lambda(\w_\alpha g)=\w_\alpha(T_{\delta_\alpha,-1}(\lambda)(g))$. We get $\lambda(\w_\beta g)=\w_\beta(T_{\delta_\beta,-1}(\lambda)(g))$ similarly.
\end{proof}

\begin{proposition}
The map $\r_L^+(\Gamma)\ra(\r^+_L)^{\psi=0}$ sending $\lambda$ to $\lambda(1+T)$ is a bijection.
\end{proposition}
\begin{proof}
See \cite[B.2.8]{P00} for a reference, where Perrin-Riou establishes a bijection from $\calE_L^\dagger(\Gamma)$ to $(\calE^\dagger_L)^{\psi=0}$ sending $\lambda$ to $\lambda(1+T)$. Her proof also works in our situation.
\end{proof}
The inverse of this map is the Mellin transformation; we denote it by $\mathrm{Mel}$. So if $g(T)\in(\r^+_L)^{\psi=0}$, then $g(T)=\mathrm{Mel}(g)(1+T)$.
\begin{lemma}
If $z=c_\alpha e_\alpha+c_\beta e_\beta\in\D^\dagger_\rig(V_{\alpha,\beta})\boxtimes\Z^\times\cap(\r^+_Le_\alpha\oplus\r^+_Le_\beta)$, then
\[
w_D(z)=\left(\begin{smallmatrix}
 0&1 \\
 1&0
\end{smallmatrix}\right)(z).
\]
Hence $\left(\begin{smallmatrix}
 0&1 \\
 1&0
\end{smallmatrix}\right)=w_D$ is an involution on $\D^\dagger_\rig(V_{\alpha,\beta})\boxtimes\Z^\times\cap(\r^+_Le_\alpha\oplus\r^+_Le_\beta)$.
\end{lemma}
\begin{proof}
By Proposition 3.7, there exist $\lambda_1,\lambda_2,\dots,\lambda_n\in\r_L(\Gamma)$ and $z_1,z_2,\dots,z_n\in(1-\varphi)\D(V_{\alpha,\beta})^{\psi=1}$ for some $n\geq 1$ such that $z=\sum_{i=1}^n \lambda_iz_i$. Since $\D(V)^{\psi=1}\subset\D(V)^\sharp$ for any $p$-adic representation $V$, we have $z_i\in\D(V_{\alpha,\beta})^\sharp\boxtimes\Z^{\times}$. Suppose $z_i=c_{\alpha,i} e_\alpha+c_{\beta,i}e_\beta$ for $1\leq i\leq n$. It follows that
\[
\sum_{i=1}^n \lambda_iz_i=\sum_{i=1}^n \lambda_i(c_{\alpha,i} e_\alpha+c_{\beta,i}e_\beta)=\sum_{i=1}^n (T_{\alpha}(\lambda_i)c_{\alpha,i} e_\alpha+T_{\beta}(\lambda_i)c_{\beta,i}e_\beta),
\]
yielding $c_\alpha=\sum_{i=1}^n T_{\alpha}(\lambda_i)c_{\alpha,i}$ and $c_\beta=\sum_{i=1}^n T_{\beta}(\lambda_i)c_{\beta,i}$. Taking Mellin transformation for latter equalities, we get
\begin{equation*}
\mathrm{Mel}(c_\alpha)(1+T)=\sum_{i=1}^n T_{\alpha}(\lambda_i)\mathrm{Mel}(c_{\alpha,i})(1+T),\hspace{1mm}\mathrm{Mel}(c_\beta)(1+T)=\sum_{i=1}^n T_{\beta}(\lambda_i)\mathrm{Mel}(c_{\beta,i})(1+T).
\end{equation*}
We conclude
\begin{equation}
\mathrm{Mel}(c_\alpha)=\sum_{i=1}^n T_{\alpha}(\lambda_i)\mathrm{Mel}(c_{\alpha,i})\quad\text{and}\quad\mathrm{Mel}(c_\beta)=\sum_{i=1}^n T_{\beta}(\lambda_i)\mathrm{Mel}(c_{\beta,i}).
\end{equation}

Following the definition of $w_D$ and Lemma 4.3, we have
\begin{equation}
\begin{split}
w_D(z)&=w_D(\sum_{i=1}^n \lambda_iz_i)=\sum_{i=1}^nT_{\delta,-1}(\lambda_i)w_D(z_i)=\sum_{i=1}^nT_{\delta,-1}(\lambda_i)(\left(\begin{smallmatrix}
 0&1 \\
 1&0
\end{smallmatrix}\right)(z_i))\\
&=(\sum_{i=1}^nT_{\alpha^{-1}\delta,-1}(\lambda_i)(\left(\begin{smallmatrix}
 0&1 \\
 1&0
\end{smallmatrix}\right)_\alpha(c_{\alpha,i})))e_\alpha+(\sum_{i=1}^nT_{\beta^{-1}\delta,-1}(\lambda_i)(\left(\begin{smallmatrix}
 0&1 \\
 1&0
\end{smallmatrix}\right)_\beta(c_{\beta,i})))e_\beta\\
&=(\sum_{i=1}^nT_{\alpha^{-1}\delta,-1}(\lambda_i)T_{\delta_{\alpha},-1}(\mathrm{Mel}(c_{\alpha,i}))(\left(\begin{smallmatrix}
 0&1 \\
 1&0
\end{smallmatrix}\right)_\alpha(1+T)))e_\alpha+\\
&\hspace{5mm}(\sum_{i=1}^nT_{\beta^{-1}\delta,-1}(\lambda_i)T_{\delta_{\beta},-1}(\mathrm{Mel}(c_{\beta,i}))(\left(\begin{smallmatrix}
 0&1 \\
 1&0
\end{smallmatrix}\right)_\beta (1+T)))e_\beta.
\end{split}
\end{equation}
On the other hand, we have
\begin{equation}
\begin{split}
\left(\begin{smallmatrix}
 0&1 \\
 1&0
\end{smallmatrix}\right)(z)&=\left(\begin{smallmatrix}
 0&1 \\
 1&0
\end{smallmatrix}\right)_\alpha(\mathrm{Mel}(c_\alpha)(1+T))e_\alpha+\left(\begin{smallmatrix}
 0&1 \\
 1&0
\end{smallmatrix}\right)_\beta(\mathrm{Mel}(c_\beta)(1+T))e_\beta\\
&=T_{\delta_\alpha,-1}(\mathrm{Mel}(c_\alpha))(\left(\begin{smallmatrix}
 0&1 \\
 1&0
\end{smallmatrix}\right)_\alpha(1+T))e_\alpha+T_{\delta_\beta,-1}(\mathrm{Mel}(c_\beta))(\left(\begin{smallmatrix}
 0&1 \\
 1&0
\end{smallmatrix}\right)_\beta(1+T))e_\beta.
\end{split}
\end{equation}
Now by $(4.5)$, we get
\[
T_{\delta_{\alpha},-1}(\mathrm{Mel}(c_\alpha))=T_{\delta_{\alpha},-1}(\sum_{i=1}^n T_{\alpha}(\lambda_i)\mathrm{Mel}(c_{\alpha,i}))=\sum_{i=1}^n T_{\alpha^{-1}\delta}(\lambda_i)T_{\delta_\alpha,-1}(\mathrm{Mel}(c_{\alpha,i}))
\]
because $\alpha\delta_\alpha=\alpha^{-1}\delta$. Similarly we have $T_{\delta_{\beta},-1}(\mathrm{Mel}(c_\beta))=\sum_{i=1}^n T_{\beta^{-1}\delta}(\lambda_i)T_{\delta_\beta,-1}(\mathrm{Mel}(c_{\beta,i}))$. We obtain the desired result by comparing $(4.6)$ and $(4.7)$.
\end{proof}

We define $\mathcal{F}_{\mathrm{an}}$ as follows. First note that
\begin{equation*}
\begin{split}
(A(\alpha)\oplus_{\pi(\beta)}A(\beta))^\ast&=\ker(A(\alpha)^\ast\oplus A(\beta)^\ast\stackrel{}{\ra}\pi(\beta)^\ast)\\
&=\Big\{(\mu_\alpha,\mu_\beta)\in A(\alpha)^\ast\oplus A(\beta)^\ast\Big|\int_{\Q}fd\mu_\beta(z)=\int_{\Q}I(f)d\mu_\alpha(z)\hspace{1mm} \text{for any} f\in \pi(\beta)\Big\}.
\end{split}
\end{equation*}
For any $(\mu_{\alpha},\mu_{\beta})\in(A(\alpha)\oplus_{\pi(\beta)}A(\beta))^\ast$,  let $c_\alpha=\A(\mu_\alpha|_{\Z}), c_\beta=\frac{1}{C(\alpha_p,\beta_p)}\A(\mu_\beta|_{\Z})$ and $c_\alpha'=\A(\w\mu_\alpha|_{\Z}), c_\beta'=\frac{1}{C(\alpha_p,\beta_p)}\A(\w\mu_\beta|_{\Z})$. Put $z_\alpha=c_\alpha e_\alpha+c_\beta e_\beta$ and  $z'_\alpha=c'_\alpha e_\alpha+c'_\beta e_\beta$. By Corollaries 3.14 and 3.16, we have $z_\alpha,z_\alpha'\in\D^\dagger_\rig(V_{\alpha,\beta})$. Since $\mathrm{Res}_{\Z^\times}z_\alpha=\w\mathrm{Res}_{\Z^\times}z_\alpha'$, we get $\mathrm{Res}_{\Z^\times}z_\alpha=w_D(\mathrm{Res}_{\Z^\times}z_\alpha')$ by Lemma 4.5. So $(z_\alpha,z_\alpha')$ is a well-defined element of $\D^\dagger_\rig(V_{\alpha,\beta})\boxtimes\mathbf{P}^1$. We define $\mathcal{F}_{\mathrm{an}}$ by setting $\mathcal{F}_{\mathrm{an}}(\mu_{\alpha},\mu_{\beta})=(z_\alpha,z_\alpha')\otimes e\in\D^\dagger_\rig(\check{V}_{\alpha,\beta})\boxtimes\mathbf{P}^1$. It is clear that $\mathcal{F}_\mathrm{an}$ is an extension of $\mathcal{F}$. Using Proposition 3.6, it is straightforward to verify that $\mathcal{F}_{\mathrm{an}}$ is $\G$-equivariant. The continuity of $\mathcal{F}_{\mathrm{an}}$ is obvious.

\subsection{Proof of Theorem 4.1}
\begin{lemma}\label{lem:dual}
$V^\ast_{\alpha,\beta}=\mathrm{Hom}(V_{\alpha,\beta},\Q)$ is isomorphic to  $V_{\beta^{-1}|x|^{k-1},\alpha^{-1}|x|^{k-1}}(1-k)$.
\end{lemma}
\begin{proof}
Note that
\[
\D_{\rm cris}(V^\ast_{\alpha,\beta})=\mathrm{Hom}_L(\D_{\rm cris}(V_{\alpha,\beta}),L)=\mathrm{Hom}_L(D(\alpha,\beta), L)
\]
as filtered $(\varphi, G_{\Q})$-modules over $L$. Let $e'_\alpha, e_\beta'\in\mathrm{Hom}_L(D(\alpha,\beta), L)$ defined by $e'_\alpha(e_\alpha)=e'_\beta(e_\beta)=1$ and $e_\alpha'(e_\beta)=e_\beta'(e_\alpha)=0$. It follows that $\D_{\rm cris}(V^\ast_{\alpha,\beta})=L\cdot e'_\alpha\oplus L\cdot e'_\beta$; the $\varphi$ and $G_{\Q}$-actions are given by $\varphi(e'_\alpha)=\alpha(p)^{-1}e_\alpha', \varphi(e'_\beta)=\beta(p)^{-1}e_\beta'$ and $\gamma(e'_\alpha)=\alpha(\chi(\gamma))^{-1}e_\alpha',\gamma(e'_\beta)=\beta(\chi(\gamma))^{-1}e_\beta'$
for any $\gamma\in\Gamma$. The filtration is given by the formula
\[
\mathrm{Fil}^i(L_n\otimes_L\mathrm{Hom}_L(D(\alpha,\beta), L))=\mathrm{Fil}^{1-i}(L_n\otimes_LD(\alpha,\beta))^{\bot}.
\]
Thus a short computation shows that for $n\geq n(V_{\alpha,\beta})$, if $\alpha\neq\beta$, then
\[
\mathrm{Fil}^i(L_n\otimes_L\D_{\rm cris}(V^\ast_{\alpha,\beta}))=\left\{
         \begin{array}{lll}
          L_n\cdot e_\alpha'\oplus L_n\cdot e_\beta' & \text{if $i\leq0$}; \\
          L_n\cdot(-e_\beta'+G(\alpha\beta^{-1})e_\alpha')  & \text{if $1\leq i\leq k-1$}; \\
          0  & \text{if $i\geq k$}.
         \end{array}
       \right.
\]
If $\alpha=\beta$, then
\[
\mathrm{Fil}^i(L_n\otimes_L\D_{\rm cris}(V^\ast_{\alpha,\beta}))=\left\{
         \begin{array}{lll}
          L_n\cdot e_\alpha'\oplus L_n\cdot e_\beta' & \text{if $i\leq0$}; \\
          L_n\cdot e_{\alpha}' & \text{if $1\leq i\leq k-1$}; \\
          0  & \text{if $i\geq k$}.
         \end{array}
       \right.
\]
Since $\beta\alpha^{-1}=\alpha^{-1}|x|^{k-1}(\beta^{-1}|x|^{k-1})^{-1}$, we immediately see that $\D_{\rm cris}(V^\ast_{\alpha,\beta}(k-1))$ is isomorphic to $D(\beta^{-1}|x|^{k-1}, \alpha^{-1}|x|^{k-1})$ as filtered $(\varphi,G_{\Q})$-modules over $L$ by mapping $-e_{\beta}'$, $e_{\alpha}'$  (with twisted actions) to $e_{\beta^{-1}|x|^{k-1}}$, $e_{\alpha^{-1}|x|^{k-1}}$, respectively. Thus $V^\ast_{\alpha,\beta}(k-1)$ is isomorphic to $V_{\beta^{-1}|x|^{k-1}, \alpha^{-1}|x|^{k-1}}$, yielding the desired result.
\end{proof}

\begin{lemma}
 Suppose $g\in\r^+_L$ and $a_i\in L$ for $1\leq i\leq l$ such that $|a_i|<1$ for every $i$ and $a_i\neq a_j$ for any $i\neq j$. Then for any $k_1,\dots,k_l\geq1$ we have
\[
\res_0(\frac{g}{\prod_{i=1}^l(T-a_i)^{k_i}}dT)=\sum_{i=1}^l\frac{1}{(k_i-1)!\prod_{j\neq i}(a_i-a_j)^{k_j}}\big((\frac{d}{dT})^{k_i-1}g\big)(a_i)
\]
\end{lemma}
\begin{proof}
For $1\leq i\leq l$ and $0\leq k\leq k_i-1$, we set
\[
b_{i,k}=\frac{1}{k!\prod_{j\neq i}(a_i-a_j)^{k_j}}\big((\frac{d}{dT})^{k-1}g\big)(a_i).
\]
Then a short computation shows that
\[
(\frac{d}{dT})^j(g-\sum_{i=1}^{k}\sum_{k=0}^{k_i-1}b_{i,k}(T-a_i)^k\prod_{j\neq i}(T-a_j)^{k_j})(a_i)=0
\]
for every $1\leq i\leq l$ and $0\leq j\leq k_i-1$. This implies that there exists an $h\in\r^+_L$ such that
$g-\sum_{i=1}^{k}\sum_{k=0}^{k_i-1}b_{i,k}(T-a_i)^k\prod_{j\neq i}(T-a_j)^{k_j}=\prod_{i=1}^l(T-a_i)^{k_i}h$. Hence
\begin{equation}
\frac{g}{\prod_{i=1}^l(T-a_i)^{k_i}}=\sum_{i=1}^{k}\sum_{k=0}^{k_i-1}\frac{b_{i,k}}{(T-a_i)^{k_i-k}}+h.
\end{equation}
Note that for $|a|<1$, we have
\[
\frac{1}{T-a}=\frac{1}{T}\cdot\frac{1}{1-a/T}=\frac{1}{T}(1+\frac{a}{T}+(\frac{a}{T})^2+\cdots).
\] So
\begin{equation}
\res_0(\frac{dT}{(T-a)^k})=\left\{
         \begin{array}{ll}
          1& \text{if $k=1$}; \\
          0& \text{if $k\geq2$}.
         \end{array}
       \right.
\end{equation}
Following $(4.8),(4.9)$, we conclude
\begin{equation}
\begin{split}
\res_0(\frac{g}{\prod_{i=1}^l(T-a_i)^{k_i}}dT)&=\res_0(\sum_{i=1}^{k}\sum_{k=0}^{k_i-1}\frac{b_{i,k}}{(T-a_i)^{k_i-k}}dT)\\
&=\sum_{i=1}^l\frac{1}{(k_i-1)!\prod_{j\neq i}(a_i-a_j)^{k_j}}\big((\frac{d}{dT})^{k_i-1}g\big)(a_i),
\end{split}
\end{equation}
yielding the desired result.
\end{proof}

\begin{lemma}
The natural map $\pi(\alpha)\ra B(\alpha)/L(\alpha)$ is injective.
\end{lemma}
\begin{proof}
In case $\alpha\neq\beta|x|$, as $\pi(\alpha)$ is irreducible and the image is dense in a nonzero space $B(\alpha)/L(\alpha)$, we conclude that $\pi(\alpha)\ra B(\alpha)/L(\alpha)$ is injective. In case $\alpha=\beta|x|$, since $\val(\alpha_p)+\val(\beta_p)=k-1$, we get $\val(\beta_p)=(k-2)/2$, yielding $k>2$. If $\pi(\alpha)\ra B(\alpha)/L(\alpha)$ is not injective, then the image must be $(\beta\circ\det)\otimes_L\mathrm{Sym}^{k-2}L^2$ because this is the only non-trivial quotient of $\pi(\alpha)$ as shown in $(1.6)$. Hence we must have $(\beta\circ\det)\otimes_L\mathrm{Sym}^{k-2}L^2=B(\alpha)/L(\alpha)$ since $(\beta\circ\det)\otimes_L\mathrm{Sym}^{k-2}L^2$ is finite dimensional and dense in $B(\alpha)/L(\alpha)$. This leads to a contradiction because $(\beta\circ\det)\otimes_L\mathrm{Sym}^{k-2}L^2$ does not posses a $\G$-invariant norm when $k>2$ (\cite[Corollary 5.1.3]{E06}).
\end{proof}
\begin{proposition}
$\{\D^\natural(V_{\alpha,\beta})\boxtimes\mathbf{P}^1,\mathcal{F}_{\mathrm{an}}((A(\alpha)\oplus_{\pi(\beta)}A(\beta))^\ast)\}_{\mathbf{P}^1}=0$.
\end{proposition}

\begin{proof}
Let $e'$ be the basis of $\delta\circ\det$ dual to $e$. Note that each element of $\D^\natural(\check{V}_{\alpha,\beta})$ is of the form $z\otimes\frac{dT}{1+T}$ for some $z\in\D^\natural(V^\ast_{\alpha,\beta})$. Since $\D^\natural(V_{\alpha,\beta})\boxtimes\mathbf{P}^1=(\D^\natural(\check{V}_{\alpha,\beta})\boxtimes\mathbf{P}^1)\otimes\delta$, it reduces to show that
\begin{equation}
\{(z\otimes\frac{dT}{1+T},z'\otimes\frac{dT}{1+T})\otimes e',\mathcal{F}_{\mathrm{an}}(\lambda_\alpha,\lambda_\beta)\}_{\mathbf{P}^1}=0
\end{equation}
for any $(\lambda_\alpha,\lambda_\beta)\in(A(\alpha)\oplus_{\pi(\beta)}A(\beta))^\ast$ and $(z,z')\in\D^\natural(V^\ast_{\alpha,\beta})\boxtimes\mathbf{P}^1$.
By Lemma \ref{lem:dual}, $V^\ast_{\alpha,\beta}$ is isomorphic to $V_{\beta^{-1}|x|^{k-1},\alpha^{-1}|x|^{k-1}}(1-k)$. Moreover, the explicit description of this isomorphism shows that there exist $c_\alpha,c_\beta$ and $c_\alpha',c_\beta'\in\r^+_L$
such that
\begin{equation*}
\begin{split}
(c_\alpha e_{\beta^{-1}|x|^{k-1}}+c_\beta e_{\alpha^{-1}|x|^{k-1}},c_\alpha'e_{\beta^{-1}|x|^{k-1}}+c_\beta'e_{\alpha^{-1}|x|^{k-1}})
\in\D^\natural(V_{\beta^{-1}|x|^{k-1},\alpha^{-1}|x|^{k-1}})\boxtimes\mathbf{P}^1,
\end{split}
\end{equation*}
and
\[
z=t^{1-k}c_\beta e_{\alpha}'-t^{1-k}c_\alpha e_{\beta}', \quad z'=t^{1-k}c_\beta' e_{\alpha}'-t^{1-k}c_\alpha' e_{\beta}'.
\]
Suppose $\mathcal{F}_{\mathrm{an}}(\lambda_\alpha,\lambda_\beta)=(d_\alpha e_\alpha+d_\beta e_\beta, d'_\alpha e_\alpha+d'_\beta e_\beta)\otimes e$. By Theorem 3.18, we may suppose
\[
(c_\alpha e_{\beta^{-1}|x|^{k-1}}+c_\beta e_{\alpha^{-1}|x|^{k-1}},c_\alpha'e_{\beta^{-1}|x|^{k-1}}+c_\beta'e_{\alpha^{-1}|x|^{k-1}})\otimes e'=\mathcal{F}(\mu_\alpha)
\]
for some $\mu_\alpha\in (B(\beta^{-1}|x|^{k-1})/L(\beta^{-1}|x|^{k-1}))^\ast$. Put $\mu_\beta=\frac{1}{C(\alpha_p,\beta_p)}\mu_{\alpha}\circ I$.
By the definition of $\{\cdot,\cdot\}_{\mathbf{P}^1}$, we have
\begin{equation}
\begin{split}
\{(z\otimes\frac{dT}{1+T},z'&\otimes\frac{dT}{1+T})\otimes e',\mathcal{F}_{\mathrm{an}}(\lambda_\alpha,\lambda_\beta)\}_{\mathbf{P}^1}=
\res_0\Big(t^{1-k}\big(\alpha(-1)c_{\beta}(\sigma_{-1}\cdot d_{\alpha})\\-\beta(-1)c_{\alpha}(\sigma_{-1}\cdot d_{\beta})
&+\alpha(-1)\varphi\psi (c_{\beta}')\varphi\psi(\sigma_{-1}\cdot d_{\alpha}')
-\beta(-1)\varphi\psi(c_{\alpha}')\varphi\psi(\sigma_{-1}\cdot d_{\beta}')\big)\frac{dT}{1+T}\Big)
\end{split}
\end{equation}
Put
\begin{equation*}
\begin{split}
S=t^{1-k}\big(\alpha(-1)c_{\beta}(\sigma_{-1}\cdot d_{\alpha})-\beta(-1)c_{\alpha}(\sigma_{-1}\cdot d_{\beta})
+\alpha(-1)\varphi\psi (c_{\beta}')\varphi\psi(\sigma_{-1}\cdot d_{\alpha}')
-\beta(-1)\varphi\psi(c_{\alpha}')\varphi\psi(\sigma_{-1}\cdot d_{\beta}')\big).
\end{split}
\end{equation*}
For any $j\geq 0$ we have
\begin{equation*}
\begin{split}
(\frac{d}{dT})^j(c_\beta(\sigma_{-1}\cdot d_{\alpha}))&=\sum_{i=0}^j{{j}\choose{i}}\Big((\frac{d}{dT})^i\int_{\Z}(1+T)^zd\mu_{\beta}(z)\Big)\Big((\frac{d}{dT})^{j-i}
\int_{\Z}(1+T)^{-z}d\lambda_{\alpha}(z)\Big)\\
&=\sum_{i=0}^jj!\int_{\Z}{{z}\choose{i}}(1+T)^{z-i}d\mu_{\beta}(z)\int_{\Z}{{-z}\choose{j-i}}(1+T)^{i-j-z}d\lambda_{\alpha}(z)\\
&=\frac{j!}{(1+T)^{j}}\sum_{i=0}^j\int_{\Z}{{z}\choose{i}}(1+T)^{z}d\mu_{\beta}(z)\int_{\Z}{{-z}\choose{j-i}}(1+T)^{-z}d\lambda_{\alpha}(z).
\end{split}
\end{equation*}
Thus for $0\leq j\leq k-2$ and $T=\eta-1$ such that $|\eta-1|<1$ we get
\begin{equation*}
\begin{split}
\Big((\frac{d}{dT})^j(c_{\beta}(\sigma_{-1}\cdot d_{\alpha}))\Big)(\eta-1)&=\frac{j!}{\eta^j}\sum_{i=0}^j\int_{\Z}{{z}\choose{i}}\eta^zd\mu_{\beta}(z)\int_{\Z}{{-z}\choose{j-i}}\eta^{-z}d\lambda_{\alpha}(z)\\
&=\frac{j!}{C(\alpha_p,\beta_p)\eta^j}\sum_{i=0}^j\int_{\Q}{{z}\choose{i}}I^{\mathrm{sm}}(1_{\Z}\cdot\eta^z)d\mu_{\alpha}(z)
\int_{\Z}{{-z}\choose{j-i}}\eta^{-z}d\lambda_{\alpha}(z)
\end{split}
\end{equation*}
We compute other terms similarly. Finally we obtain
\begin{equation}
\begin{split}
\Big((\frac{d}{dT})^jS\Big)(\eta-1)=&\frac{j!}{C(\alpha_p,\beta_p)\eta^j}\sum_{i=0}^j(\alpha(-1)\int_{\Q}{{z}\choose{i}}I^{\mathrm{sm}}(1_{\Z}\cdot\eta^z)d\mu_{\alpha}(z)
\int_{\Z}{{-z}\choose{j-i}}\eta^{-z}d\lambda_{\alpha}(z)\\
-&\beta(-1)\int_{\Z}{{z}\choose{i}}\eta^{z}d\mu_{\alpha}(z)\int_{\Z}{{-z}\choose{j-i}}I^{\mathrm{sm}}(1_{\Z}\cdot\eta^{-z})d\lambda_{\alpha}(z)\\
+&\alpha(-1)\int_{\Q}\w({{z}\choose{i}}I^{\mathrm{sm}}(1_{p\Z}\cdot\eta^{z}))d\mu_{\alpha}(z)\int_{\Q}\w ({{-z}\choose{j-i}}1_{p\Z}\cdot\eta^{-z})d\lambda_{\alpha}(z)\\
-&\beta(-1)\int_{\Q}\w({{z}\choose{i}}1_{p\Z}\cdot\eta^{z})d\mu_{\alpha}(z)\int_{\Q}\w({{-z}\choose{j-i}}
I^{\mathrm{sm}}(1_{p\Z}\cdot\eta^{-z}))d\lambda_{\alpha}(z)).
\end{split}
\end{equation}
For $m\geq m(V_{\alpha,\beta})+1$ and $n=0,1$, applying Lemma 1.6, we get that
\begin{equation*}
\begin{split}
I^{\mathrm{sm}}(1_{p^n\Z}\cdot\eta_{p^m}^{\pm{z}})&=C(\alpha_p,\beta_p)(\frac{\beta_p}{p\alpha_p})^mG(\beta^{-1}\alpha,\eta_{p^{m}}^{\pm1})(1_{p^n\Z}\cdot\eta_{p^m}^{\pm{z}})\\
&=C(\alpha_p,\beta_p)(\frac{\beta_p}{p\alpha_p})^m\beta\alpha^{-1}(\pm1)G(\beta^{-1}\alpha,\eta_{p^{m}})(1_{p^n\Z}\cdot\eta_{p^m}^{\pm{z}}).
\end{split}
\end{equation*}
Let $\eta=\eta_{p^m}$, by $(4.13)$, we get that $(\frac{d}{dT})^jS(\eta_{p^m}-1)=0$ for $0\leq j\leq k-2$ and $m\geq m(V_{\alpha,\beta})$+1.

Let $q=\varphi(T)/T$. Recall that $t=T\cdot(q/p)\cdot(\varphi(q)/p)\cdot(\varphi^2(q)/p)\cdots.$
The roots of
\[
\varphi^n(q)/p=((1+T)^{p^{n+1}}-1)/(p((1+T)^{p^n}-1))
\]
are $\mu_{p^{n+1}}\setminus\mu_{p^{n}}$. Let $t'=\prod_{n\geq m(V_{\alpha,\beta})}(\varphi^n(q)/p)$. Since $\big((\frac{d}{dT})^jS\big)(\eta_{p^m}-1)=0$ for $0\leq j\leq k-2$ and $m\geq m(V_{\alpha,\beta})$+1,
we conclude that $(t')^{k-1}$ divides $S$ in $\r^+_L$; we denote by $S'$ the quotient.

The right hand side of $(4.12)$ is equal to
 \[
 \res_0(\frac{S}{t^{k-1}(1+T)})=\res_0(\frac{p^{(k-1)(m(V_{\alpha,\beta})-1)}S'}
 {\prod_{\eta^{p^{m(V_{\alpha,\beta})}}=1}(T+1-\eta)^{k-1}}).
 \]
Applying Lemma $4.7$, we get that
\begin{equation}
\begin{split}
\res_0(\frac{p^{(k-1)(m(V_{\alpha,\beta})-1)}S'}{\prod_{\eta^{p^{m(V_{\alpha,\beta})}}=1}(T+1-\eta)^{k-1}})&=
\sum_{\eta^{p^{m(V_{\alpha,\beta})}}=1}\frac{p^{(k-1)(m(V_{\alpha,\beta})-1)}\big((\frac{d}{dT})^{k-2}\frac{S'}{1+T}\big)(\eta-1)}{(k-2)!
\prod_{\eta'^{p^{m(V_{\alpha,\beta})}}=1,\eta'\neq\eta}(\eta-\eta')^{k-1}}\\
&=\sum_{\eta^{p^{m(V_{\alpha,\beta})}}=1}\frac{\eta^{k-1}\big((\frac{d}{dT})^{k-2}\frac{S'}{1+T}\big)(\eta-1)}{p^{k-1}(k-2)!},
\end{split}
\end{equation}
where we used $\prod_{\eta'^{p^{m(V_{\alpha,\beta})}}=1,\eta'\neq\eta}(\eta-\eta')=(\frac{d}{dT}T^{p^{m(V_{\alpha,\beta})}})(\eta)
=\frac{p^{m(V_{\alpha,\beta})}}{\eta}$.

The last line of $(4.14)$ is equal to
\begin{equation}
\begin{split}
&\sum_{\eta^{p^{m(V_{\alpha,\beta})}}=1}\frac{\eta^{k-1}}{p^{k-1}(k-2)!}\Big(\sum_{j=0}^{k-2}{{k-2}\choose{j}}
\Big((\frac{d}{dT})^jS'\Big)\frac{(-1)^{k-2-j}}{(1+T)^{k-1-j}}\Big)(\eta-1)\\
=&\sum_{\eta^{p^{m(V_{\alpha,\beta})}}=1}\sum_{j=0}^{k-2}\frac{(-1)^{k-j}\eta^{j}}{p^{k-1}j!(k-2-j)!}
\Big((\frac{d}{dT})^jS'\Big)(\eta-1)\\
=&\sum_{\eta^{p^{m(V_{\alpha,\beta})}}=1}(\sum_{j=0}^{k-2}\sum_{i=0}^j\frac{(-1)^{k-j}\eta^{j}}{p^{k-1}i!(j-i)!(k-2-j)!}
\Big((\frac{d}{dT})^{j-i}S\Big)\Big(\frac{d}{dT})^{i}(t')^{1-k}\Big)(\eta-1)
\end{split}
\end{equation}
A short computation shows that for any $i\geq0$, there exists a $c(i)\in\Z$ such that
$\big((\frac{d}{dT})^i(t')^{1-k}\big)(\eta-1)=c(i)\eta^{-i}$ for any $\eta\in\mathrm{\mu}_{p^{m(V_{\alpha,\beta})}}$.
Thus the last line of $(4.15)$ is equal to
\begin{equation}
\begin{split}
&\sum_{\eta^{p^{m(V_{\alpha,\beta})}}=1}\sum_{j=0}^{k-2}\sum_{i=0}^j\frac{(-1)^{k-j}\eta^{i}c(j-i)}{p^{k-1}i!(j-i)!(k-2-j)!}
\Big((\frac{d}{dT})^{i}S\Big)(\eta-1)
\end{split}
\end{equation}
Put $C(i,j)=\frac{(-1)^{k-j}c(j-i)}{C(\alpha,\beta)p^{k-1}(j-i)!(k-2-j)!}$. Then $(4.16)$ is equal to
\begin{equation}
\begin{split}
\sum_{\eta^{p^{m(V_{\alpha,\beta})}}=1}&\sum_{j=0}^{k-2}\sum_{i=0}^jC(i,j)
\sum_{h=0}^i((\alpha(-1)\int_{\Q}{{z}\choose{h}}I^{\mathrm{sm}}(1_{\Z}\cdot\eta^z)d\mu_{\alpha}(z)
\int_{\Z}{{-z}\choose{i-h}}\eta^{-z}d\lambda_{\alpha}(z)\\
-&\beta(-1)\int_{\Z}{{z}\choose{h}}\eta^{z}d\mu_{\alpha}(z)\int_{\Z}{{-z}\choose{i-h}}I^{\mathrm{sm}}(1_{\Z}\cdot\eta^{-z})d\lambda_{\alpha}(z)\\
+&\alpha(-1)\int_{\Q}\w({{z}\choose{h}}I^{\mathrm{sm}}(1_{p\Z}\cdot\eta^{z}))d\mu_{\alpha}(z)\int_{\Q}\w ({{-z}\choose{i-h}}1_{p\Z}\cdot\eta^{-z})d\lambda_{\alpha}(z)\\
-&\beta(-1)\int_{\Q}\w({{z}\choose{h}}1_{p\Z}\cdot\eta^{z})d\mu_{\alpha}(z)\int_{\Q}\w({{-z}\choose{j-i}}
I^{\mathrm{sm}}(1_{p\Z}\cdot\eta^{-z}))d\lambda_{\alpha}(z)).
\end{split}
\end{equation}

We now prove that $(4.17)$ is equal to $0$. Let $Y$ be the $L$-vector space generated by all the ${{-z}\choose{i-h}}\eta^{-z}, {{-z}\choose{i-h}}I^{\mathrm{sm}}(1_{\Z}\cdot\eta^{-z})$, $\w({{-z}\choose{i-h}}1_{p\Z}\cdot\eta^{-z})$ and $\w({{-z}\choose{i-h}}I^{\mathrm{sm}}(1_{p\Z}\cdot\eta^{-z}))$ for all $0\leq h\leq i\leq k-2$ and $\eta\in\mu_{p^{m(V_{\alpha,\beta})}}$. Let $e=\{f_1(z),\cdots,f_n(z)\}$ be an $L$-basis of $Y$. We expand $(4.17)$ in the form
\begin{equation}
\sum_{m=1}^{n}\int_{\Q}g_m(z)d\mu_\alpha(z)\int_{\Q}f_m(z)d\lambda_\alpha(z)
\end{equation}
for some $g_m(z)\in\pi(\beta^{-1}|x|^{k-1})$. We claim that all the terms $\int_{\Q}g_m(z)d\mu_\alpha(z)$ are zero. In fact, for any $1\leq m_0\leq n$, since $Y\subseteq \pi(\alpha)\subseteq B(\alpha)/L(\alpha)$, applying Hahn-Banach theorem to the Banach space $B(\alpha)/L(\alpha)$, we pick $\lambda_\alpha'\in (B(\alpha)/L(\alpha))^\ast$ such that $\int_{\Q}f_m(z)d\lambda_\alpha'(z)\neq0$ if and only if $m=m_0$. Let $\lambda_\beta'=\lambda_\alpha'\circ\widehat{I}$. By Theorem 3.18, we know $\mathcal{F}(\lambda_{\alpha}',\lambda_\beta')\in \D^\natural(\check{V}_{\alpha,\beta})\boxtimes\mathbf{P}^1$. Hence $\{(z\otimes\frac{dT}{1+T},z'\otimes\frac{dT}{1+T})\otimes e', \mathcal{F}(\lambda_\alpha',\lambda_\beta')\}_{\mathbf{P}^1}=0$ by Theorem 3.9. This implies
\[
0=\sum_{m=1}^{n}\int_{\Q}g_m(z)d\mu_\alpha(z)\int_{\Q}f_m(z)d\lambda_\alpha'(z)=\int_{\Q}g_{m_0}(z)d\mu_\alpha(z)\int_{\Q}f_{m_0}(z)d\lambda_\alpha'(z),
\]
yielding $\int_{\Q}g_{m_0}(z)d\mu_\beta(z)=0$. We conclude that $(4.17)$ is zero.
\end{proof}
\begin{remark}
Although it is not difficult to compute each integral appearing in $(4.17)$, it looks very difficult to show that $(4.17)$ is equal to zero by a direct computation. Here we show that $(4.17)$ is zero by Theorem 3.18, which is actually proved by some topological argument (see \cite{C08} for more details).
\end{remark}

\noindent $\mathrm{Proof\hspace{1mm}of\hspace{1mm}Theorem\hspace{1mm}4.1}$:
By Proposition 4.9, we have $\mathcal{F}_{\mathrm{an}}((A(\alpha)\oplus_{\pi(\beta)}A(\beta))^\ast)
\subseteq\D^{\natural}_{\rig}(\check{V}_{\alpha,\beta})
\boxtimes\mathbf{P}^1$ because $\D^{\natural}_{\rig}(\check{V}_{\alpha,\beta})\boxtimes\mathbf{P}^1$ is the orthogonal complement of $\D^\natural(\check{V}_{\alpha,\beta})\boxtimes\mathbf{P}^1$. By Theorem 3.10(i), we have $\D^{\natural}_{\rig}(\check{V}_{\alpha,\beta})\boxtimes\mathbf{P}^1=(\Pi(V_{\alpha,\beta})_{\mathrm{an}})^\ast$. So $(4.1)$ implies the following commutative diagram

\[
\xymatrix{
(B(\alpha)/L(\alpha))^\ast\ar^{(i\circ i_{\alpha,\beta})^\ast}[d] \ar^{\mathcal{F}}[r]& \Pi(V_{\alpha,\beta})^\ast\ar[d] \\
 (A(\alpha)\oplus_{\pi(\beta)}A(\beta))^\ast \ar^{\mathcal{F}_\mathrm{an}}[r]&(\Pi(V_{\alpha,\beta})_{\mathrm{an}})^\ast.}
\]
From Proposition 1.17, we get that $\mathcal{F}_{\mathrm{an}}\circ i_{\alpha,\beta}^\ast$ is an isomorphism because $\mathcal{F}$ is an isomorphism. By the construction of $\mathcal{F}_{\mathrm{an}}$, it is clear that $\mathcal{F}_{\mathrm{an}}$ is injective. We conclude that both $i_{\alpha,\beta}^\ast$ and $\mathcal{F}_{\mathrm{an}}$ are isomorphisms. Therefore $i_{\alpha,\beta}^\ast$ is a topological isomorphism because the topology of coadmissible modules are canonical, yielding that $i_{\alpha,\beta}$ is a topological isomorphism.
\begin{remark}
Note that the mere existence of $(4.1)$ already implies that $i_{\alpha,\beta}^\ast$ is injective. In fact, by $(4.1)$, we see that $\mathcal{F}_{\mathrm{an}}\circ i_{\alpha,\beta}^\ast$ maps $((B(\alpha)/L(\alpha))_{\mathrm{an}})^\ast$ one-to-one onto $\D^{\natural}_{\rig}(\check{V}_{\alpha,\beta})\boxtimes\mathbf{P}^1$. This yields that $i_{\alpha,\beta}^\ast$ is injective. One can also prove the surjectivity of $i_{\alpha,\beta}^\ast$ by results from representation theory (\cite[Corollaires 5.3.6, 5.4.3]{BB06}). Our treatment here is completely different. We actually prove the surjectivity of $i_{\alpha,\beta}^\ast$ by Proposition 4.9. The advantage of our method is that the way of proving Proposition 4.9 is quite general. One can adapt it to prove similar results in other cases.
\end{remark}

\section{Computation of Jacquet modules}
In \cite{E06b}, Emerton introduced the notation of locally analytic Jacquet modules. Recall that if $W$ is a locally analytic $\G$-representation of compact type, then the Jacquet module $J_\mathrm{B(\Q)}(W)$ is a certain locally analytic representation of $\mathrm{T}(\Q)$ over $L$ functorially associated to $W$. We do not recall the definition here (see \cite{E06b} for more details). But do recall that $J_\mathrm{B(\Q)}(U)$ is additive and left exact. In this section, we will prove \cite[Conjecture 3.3.1(8)]{E06} for that $V\in\mathscr{S}^{\mathrm{cris}}_*$ which are not exceptional. We learned this proof from Emerton. We are grateful to him for allowing us to put it in this paper.
\begin{proposition}
If $\alpha\neq\beta$, then $J_\mathrm{B(\Q)}(\mathrm{B}(V_{\alpha,\beta})_{\mathrm{an}})=L\cdot(x^{k-2}\beta\otimes\alpha|x|^{-1})\oplus L\cdot(x^{k-2}\alpha\otimes\beta|x|^{-1})$.
\end{proposition}
\begin{proof}
Note that there is a short exact sequence
\[
0\longrightarrow\pi(\beta)\longrightarrow A(\beta)\longrightarrow(\Ind^{\G}_{\mathrm{B}(\Q)}x^{k-1}\beta\otimes \alpha (x|x|)^{-1})^{\mathrm{an}}\longrightarrow0.
\]
(This follows from the short exact sequence $(*)$ on p.$123$ of \cite{ST01}.) It follows from Theorem $4.1$ that $\mathrm{B}(V_{\alpha,\beta})_{\mathrm{an}}$ fits into the following short exact sequence
\begin{equation}
0\longrightarrow(\Ind^{\G}_{\mathrm{B}(\Q)}\alpha\otimes x^{k-2}\beta|x|^{-1})^{\mathrm{an}}\longrightarrow \mathrm{B}(V_{\alpha,\beta})_{\mathrm{an}}\longrightarrow(\Ind^{\G}_{\mathrm{B}(\Q)}x^{k-1}\beta\otimes \alpha (x|x|)^{-1})^{\mathrm{an}}\longrightarrow0.
\end{equation}
Applying the functor $J_\mathrm{B(\Q)}$ to $(5.1)$, we obtain a short exact sequence
\begin{equation}
0\longrightarrow J_{\mathrm{B(\Q)}}((\Ind^{\G}_{\mathrm{B}(\Q)}\alpha\otimes x^{k-2}\beta|x|^{-1})^{\mathrm{an}})\longrightarrow J_{\mathrm{B}(\Q)}(\mathrm{B}(V_{\alpha,\beta})_{\mathrm{an}})\longrightarrow J_{\mathrm{B}(\Q)}((\Ind^{\G}_{\mathrm{B}(\Q)}x^{k-1}\beta\otimes\alpha(x|x|)^{-1})^{\mathrm{an}})
\end{equation}
because Jacquet functor is left exact. By \cite[Proposition 5.2.1(1),(3),(4)]{E06b}, we get
\[
J_{\mathrm{B}(\Q)}((\Ind^{\G}_{\mathrm{B}(\Q)}\alpha\otimes x^{k-2}\beta|x|^{-1})^{\mathrm{an}})
=L\cdot(x^{k-2}\beta\otimes\alpha|x|^{-1})\oplus L\cdot(x^{k-2}\alpha\otimes\beta|x|^{-1})
\]
and
\[
J_{\mathrm{B}(\Q)}((\Ind^{\G}_{\mathrm{B}(\Q)}x^{k-1}\beta\otimes \alpha(x|x|)^{-1})^{\mathrm{an}})
=L\cdot (x^{-1}\alpha\otimes x^{k-1}\beta|x|^{-1}).
\]
We claim that the map from the middle term to the last term in $(5.2)$ vanishes. It is obvious that this claim yields the desired result. We will prove this claim in the rest. If the claim is not true, then the map from the middle term to the last term in $(5.2)$ must be surjective since $J_{\mathrm{B}(\Q)}((\Ind^{\G}_{\mathrm{B}(\Q)}x^{k-1}\beta\otimes \alpha(x|x|)^{-1})^{\mathrm{an}})$ is only $1$-dimensional. So there must be an inclusion $L\cdot(x^{-1}\alpha\otimes x^{k-1}\beta|x|^{-1})\hookrightarrow J_{\mathrm{B}(\Q)}(\mathrm{B}(V_{\alpha,\beta})_{\mathrm{an}})$ because the character $x^{-1}\alpha\otimes x^{k-1}\beta|x|^{-1}$ does not appear in $J_{\mathrm{B}(\Q)}((\Ind^{\G}_{\mathrm{B}(\Q)}\alpha\otimes x^{k-2}\beta|x|^{-1})^{\mathrm{an}})$. It follows from \cite[Theorem 5.2.5]{E06} that this inclusion leads to a map
\[
(\Ind^{\G}_{\mathrm{B}(\Q)}x^{k-1}\beta\otimes\alpha(x|x|)^{-1})^{\mathrm{an}}\ra\mathrm{B}(V_{\alpha,\beta}
)_{\mathrm{an}}
\]
which would split the exact sequence $(5.1)$. However, by \cite[Lemma 6.7.4]{E06}, we know that $(5.1)$ is nonsplit, yielding a contradiction.
\end{proof}
We next recall some notations introduced in \cite{E06}. In the following, let $V$ be a $2$-dimensional $L$-linear representation of $G_{\Q}$.
\begin{definition}
A refinement of $V$ is a triple $R=(\eta,c,r)$, where:
\begin{enumerate}
\item[(i)]$\eta$ is a continuous character $G_{\Q}\ra L^\times$ such that $V(\eta^{-1})$ has at least one Hodge-Tate weight equal to zero;
\item[(ii)]$c\in L^\times$;
\item[(iii)]$r$ is a nonzero $G_{\Q}$-equivariant $L$-linear map $V^{\ast}(\eta)\ra (L\otimes_{\Q}\bcris)^{\varphi=c}$.
\end{enumerate}
Note that we may regard $r$ as a nonzero element of $\D^+_{\mathrm{cris}}(V(\eta^{-1}))^{\varphi=c}$. We say that a pair of refinements $R_1=(\eta_1,c_1,r_1)$ and $R_2=(\eta_2,c_2,r_2)$ are equivalent if there exists $c'\in\OO_L^\times$ and $0\neq x\in (L\otimes_{\Q}{\Q^{\mathrm{ur}}})^{\varphi=c'}$ such that $r_2=xr_1$ (and hence such that $\eta_2=\eta_1\mathrm{ur}(c'^{-1})$ and $c_2=c'c_1$). Let $[R]$ denote the equivalence class of refinements which $R$ belongs to.
\end{definition}
\begin{definition}
If $R=(\eta,c,r)$ is a refinement of $V$, then we define the associated abelian Weil group representation to be the map
$\sigma(R):\mathrm{W}_{\Q}^{\mathrm{ab}}\cong\Q^\times\ra \mathrm{T}(L)$ defined via the characters $(\eta\mathrm{ur}(c),(\det V)\eta^{-1}\mathrm{ur}(c^{-1}))$. If $R'$ is equivalent to $R$, then it is clear to see that $\sigma(R')=\sigma(R)$.
\end{definition}
\begin{remark}
One can show that $V$ has at least one refinement if and only if $V$ is a trianguline representation. In fact, suppose that $R=(\eta,c,r)$ is a refinement of $V$. We regard $r$ as an element of $\D^+_{\mathrm{cris}}(V(\eta^{-1}))\subseteq(\D_\rig^\dagger(V(\eta^{-1})))^{\Gamma}$. Let $M$ be the saturation of the rank $1$ $\m$-submodule $\r_L r$ in $\D_\rig^\dagger(V(\eta^{-1}))$. Twisting the short exact sequence
\[
0\longrightarrow M\longrightarrow\D_\rig^\dagger(V(\eta^{-1}))\longrightarrow \D_\rig^\dagger(V(\eta^{-1}))/M \longrightarrow0.
\]
of $\m$-modules with $\r_L(\eta)$, we obtain a triangulation
\[
0\longrightarrow M(\eta)\longrightarrow\D_\rig^\dagger(V)\longrightarrow\D_\rig^\dagger(V)/(M(\eta))\longrightarrow0.
\]
of $\D_\rig^\dagger(V)$. Conversely, suppose that $\D_\rig^\dagger(V)$ has a triangulation
\[
0\longrightarrow \r_L(\delta_1)\longrightarrow\D_\rig^\dagger(V)\longrightarrow \r_L(\delta_2)\longrightarrow0.
\]
Let $\eta:G_{\Q}\ra L^\times$ be the character defined by $\eta(g)=\delta_1(\chi(g))$, $c=\delta_1(p)$, and let $r$ be a nonzero element of $(\r_L(\delta_1\eta^{-1}))^{\Gamma}$. We get a refinement $R=(\eta,c,r)$ of $V$.
\end{remark}
\begin{definition}
Let $\mathrm{Ref}(V)$ denote the set of equivalence classes of refinements of $V$. For any $\sigma\in\mathrm{Hom}_{\mathrm{cont}}(\mathrm{W}_{\Q}^{\mathrm{ab}},\mathrm{T}(L))$, set
$\mathrm{Ref}^{\sigma}(V)=\{[R]|\sigma(R)=\sigma\}$.
\end{definition}
\noindent If we fix $\sigma$, then it is not difficult to see that $\mathrm{Ref}^{\sigma}(V)$ is either empty or a point, except in the case $V=\eta\oplus\eta$ and $\sigma=\eta\otimes\eta$ where $\mathrm{Ref}^{\sigma}(V)\cong\mathrm{P}^1(L)$. Thus we regard $\mathrm{Ref}^{\sigma}(V)$ as projective space over $L$ and denote its dimension by $\dim\mathrm{Ref}^{\sigma}(V)$.
\begin{definition}
Let $W$ be a compact type locally analytic $\G$-representation over $L$.
\begin{enumerate}
\item[(i)]Define $\mathrm{Exp}(W)$ to be the set of $1$-dimensional $\mathrm{T}(\Q)$-invariant subspaces of $J_{\mathrm{B}(\Q)}(W)$.
\item[(ii)]For any line $l\in\mathrm{Exp}(W)$, write $\delta(l)\in\mathrm{Hom}_{\mathrm{cont}}(\mathrm{T}(\Q),L^\times)$ to denote the character via which $\mathrm{T}(\Q)$ acts on $l$.
\item[(iii)]For any $\delta\in\mathrm{Hom}_{\mathrm{cont}}(\mathrm{T}(\Q),L^\times)$, write $\mathrm{Exp}^{\delta}(W):=\{l\in\mathrm{Exp}(W)|\delta(l)=\delta\}$
\end{enumerate}
\end{definition}
\noindent If we fix a character $\chi$, then $\mathrm{Exp}^{\delta}(W)$ has the structure of a projective space, namely it is the projectivization of the $\chi$-eigenspace $J_{\mathrm{B}(\Q)}^{\chi}(W)$. We then define $\dim \mathrm{Exp}^{\delta}(W)$ to be the dimension of this projective space.

We identify $\mathrm{Hom}_{\mathrm{cont}}(\mathrm{T}(\Q),L^\times)$ with $\mathrm{Hom}_{\mathrm{cont}}(\mathrm{W}_{\Q}^{\mathrm{ab}},\mathrm{T}(L))$ via the isomorphism $\Q^\times\cong\mathrm{W}_{\Q}^{\mathrm{ab}}$ provided by the local Artin map. The following corollary verifies \cite[Conjecture 3.3.1(8)]{E06}, which relates the space of refinements of $V$ and Jacquet modules of $\mathrm{B}(V)_{\mathrm{an}}$, in the case when $V\in\mathscr{S}^{\mathrm{cris}}_*$ is not exceptional. Let us remind the reader that our normalization of the $p$-adic local Langlands correspondence for $\G$ differs by a twist of $(x|x|)^{-1}\circ \det$ from the normalization chosen by Emerton as explained in subsection 3.1. So the right hand side of $(5.3)$ is $\dim\mathrm{Exp}^{\eta|x|\otimes x\psi }(\mathrm{B}(V)_{\mathrm{an}}\otimes(x|x|\circ\det))$ instead of $\dim\mathrm{Exp}^{\eta|x|\otimes x\psi}(\mathrm{B}(V)_{\mathrm{an}})$ in Emerton's formulation.
\begin{corollary}
Keep notations as above. If $V\in\mathscr{S}^{\mathrm{cris}}_*$ is not exceptional, then
\begin{equation}
\dim \mathrm{Ref}^{\eta\otimes\psi}(V)=\dim\mathrm{Exp}^{\eta|x|\otimes x\psi }(\mathrm{B}(V)_{\mathrm{an}}\otimes(x|x|\circ\det)).
\end{equation}
for any $\eta\otimes\psi\in\mathrm{Hom}_{\mathrm{cont}}(\mathrm{T}(\Q),L^\times)$ .
\end{corollary}
\begin{proof}
Since $V\in\mathscr{S}^{\mathrm{cris}}_*$, we get that $V=V_{\alpha,\beta}(\delta)$ for some pair $(\alpha,\beta)$ and $\delta\in\mathrm{Hom}_{\mathrm{cont}}(G_{\Q},L^\times)$. Furthermore, the condition that $V$ is not exceptional implies that $V_{\alpha,\beta}$ is not exceptional, yielding $\alpha\neq\beta$. It suffices to prove the corollary for $V_{\alpha,\beta}$. By Proposition 5.1, we first have
\[
\dim\mathrm{Exp}^{\eta|x|\otimes x\psi}(\mathrm{B}(V_{\alpha,\beta})_{\mathrm{an}}\otimes(x|x|\circ\det))=\left\{
         \begin{array}{lll}
          0& \text{if $(\eta,\psi)=(x^{k-1}\beta,\alpha)$}; \\
          0& \text{if $(\eta,\psi)=(x^{k-1}\alpha,\beta)$}; \\
          -1& \text{otherwise}.
         \end{array}
       \right.
\]
On the other hand, by the construction of $D(\alpha,\beta)$, it is clear to see that
\[
\D_{\mathrm{cris}}^+(V_{\alpha,\beta}(1-k))^{\varphi=\alpha(p)p^{k-1}}=L\cdot e_\alpha
\]
Therefore $R_\alpha=(\chi^{k-1},\alpha(p)p^{k-1},e_\alpha)$ is a refinement of $V$. Similarly, $R_\beta=(\chi^{k-1},\beta(p)p^{k-1},e_\beta)$ is also a refinement of $V_{\alpha,\beta}$. A straightforward computation shows that
\[
\sigma(R_\alpha)=x^{k-1}\beta\otimes \alpha\hspace{1mm}\text{and}\hspace{1mm}\sigma(R_\beta)=x^{k-1}\alpha\otimes \beta.
\]
By \cite[Proposition 4.2.4]{E06}, we know that $V_{\alpha,\beta}$ has only two inequivalent refinements. Since $\sigma(R_\alpha)\neq\sigma(R_\beta)$, we conclude that $R_\alpha$ and $R_\beta$ are exactly all the inequivalent refinements of $V$. It follows that
\[
\dim\mathrm{Ref}^{\eta|x|\otimes x\psi}(V_{\alpha,\beta})=\left\{
         \begin{array}{lll}
          0& \text{if $(\eta,\psi)=(x^{k-1}\beta,\alpha)$}; \\
          0& \text{if $(\eta,\psi)=(x^{k-1}\alpha,\beta)$}; \\
          -1& \text{otherwise}.
         \end{array}
       \right.
\]
yielding the desired result.
\end{proof}
\begin{remark}
The result of Corollary $5.7$ also follow from \cite[Proposition 6.6.5]{E06}. In fact, the assumption on locally algebraic vectors in \cite[Proposition 6.6.5(2)]{E06} is now proved by Colmez \cite{C08}. The dimension of the left hand side of the inequality in \cite[Proposition 6.6.5(3)]{E06} is always $-1$ for our $V$, and so that inequality becomes an equality.
\end{remark}
\section*{Acknowledgements}
It is a pleasure to thank Christophe Breuil, Pierre Colmez, Matthew Emerton and Benjamin Schraen for very helpful discussions and correspondences during the preparation of this paper. Thanks to Christophe Breuil, Florian Herzig and Liang Xiao for useful comments on early drafts of this paper. Thanks to my friend Chenyang Xu for his constant encouragement. The author is grateful to Matthew Emerton for teaching him the content of section 5. The author would like to express his great gratitude to Marie-France Vigneras for arranging his visit to Institut de Math\'ematiques de Jussieu
during spring 2008 and for being his tutor for the postdoctoral research program of Foundation Sciences Math\'ematiques de Paris. This paper could not exist without her help. The author wrote this paper as a post-doc of Foundation Sciences Math\'ematiques de Paris at the Institut de Math\'ematiques de Jussieu. Part of this work were done while the author was a visitor at Beijing International Center for Mathematical Research. The author is grateful to these institutions for their hospitality. Finally the author would like to acknowledge his overwhelming debt to the works of Berger-Breuil and Colmez.

\end{document}